\documentclass[10pt,a4paper]{article}

\usepackage[english]{babel}
\usepackage{fancyhdr}
\usepackage[latin1]{inputenc}
\usepackage{amsmath}
\usepackage{amsfonts}
\usepackage{amssymb}
\usepackage{graphicx}
\usepackage{latexsym}
\usepackage{amsthm}
\usepackage{color}
\usepackage{midpage}
\usepackage[width=16.00cm]{geometry}
\usepackage{amscd}
\usepackage{tikz-cd}
\usepackage{hyperref}

\numberwithin{equation}{section}
\theoremstyle{plain} 
\newtheorem{thm}{Theorem}[section]

\newtheorem{lem}[thm]{Lemma}
\newtheorem{prop}[thm]{Proposition}
\newtheorem{ex}[thm]{Example}
\theoremstyle{definition} 
\theoremstyle{remark} 
\newtheorem{oss}{Remark}
\newcommand{\eps}{\varepsilon}
\providecommand{\keywords}[1]{\textbf{Keywords and Phrases:} #1}
\providecommand{\classification}[1]{\textbf{AMS Classification:}#1}

\DeclareMathOperator{\Var}{Var}

\newtheorem{ass}{Assumption}

\begin{document}

\author{{Anna Paola Todino}
\\
Gran Sasso Science Institute, L'Aquila\\
\small annapaola.todino@gssi.it}
\date{}
\title{A Quantitative Central Limit Theorem for the Excursion Area of Random
Spherical Harmonics over Subdomains of $\mathbb{S}^{2}$ }
\maketitle

\begin{abstract}
\noindent In recent years, considerable interest has been drawn by the
analysis of geometric functionals for the excursion sets of random
eigenfunctions on the unit sphere (spherical harmonics). In this paper, we
extend those results to proper subsets of the sphere $\mathbb{S}^{2}$, i.e.,
spherical caps, focussing in particular on the excursion area. Precisely, we show that the asymptotic behaviour of the excursion area is
dominated by the so-called second-order chaos component, and we exploit this
result to establish a Quantitative Central Limit Theorem, in the high energy
limit. These results generalize analogous findings for the full sphere;
their proofs, however, requires more sophisticated techniques, in particular
a careful analysis (of some independent interest) for smooth approximations
of the indicator function for spherical caps subsets.
\end{abstract}

\begin{itemize}
\item \keywords{} Gaussian Eigenfunctions, Spherical Harmonics, Excursion Area,
Quantitative Central Limit Theorem, Wiener-chaos expansion, Clebsch-Gordan coefficients.

\item \classification{} 42C10, 33C55, 60B10.
\end{itemize}

\section{Introduction and background results}
Let $\mathbb{S}^{2}$ be the unit 2-dimensional sphere and consider the
Helmholtz equation 
\begin{equation*}
\Delta _{\mathbb{S}^{2}}T_{\ell }+\lambda _{\ell }T_{\ell }=0,\mbox{ }\mbox{
}\mbox{ }T_{\ell }:\mathbb{S}^{2}\rightarrow \mathbb{R},
\end{equation*}%
where $\Delta _{\mathbb{S}^{2}}$ is the Laplace-Beltrami operator on $%
\mathbb{S}^{2}$, defined as usual as
\begin{equation*}
\dfrac{1}{\sin \theta }\dfrac{\partial }{\partial \theta }\bigg\{\sin \theta 
\dfrac{\partial }{\partial \theta }\bigg\}+\dfrac{1}{\sin ^{2}\theta }\dfrac{%
\partial ^{2}}{\partial \varphi ^{2}},\mbox{ }0\leq \theta \leq \pi ,\mbox{ }%
0\leq \varphi < 2\pi ,
\end{equation*}%
and $\lambda _{\ell }=\ell (\ell +1),\mbox{ }\ell =0,1,\dots $. For a given
eigenvalue $\lambda _{\ell },$ the corresponding eigenspace is the $(2\ell
+1)-$dimensional space of spherical harmonics of degree $\ell .$ A standard,
complex-valued $L^{2}$ basis $\{Y_{\ell m}(\cdot )\}_{m=-\ell ,\dots ,\ell }$
can be defined as (see \cite{M e Peccati} p. 64)%
\begin{eqnarray*}
&Y_{\ell m}(\theta ,\varphi ) &:=\sqrt{\frac{2\ell +1}{4\pi }}\sqrt{\frac{%
(\ell -m)!}{(\ell +m)!}}P_{\ell m}(\cos \theta )\exp (im\varphi )\text{ ,
for }m\geq 0\text{ ,} \\
&Y_{\ell m}(\theta ,\varphi ) &:=(-1)^{m}\overline{Y}_{\ell ,-m}(\theta
,\varphi )\text{ , for }m<0\text{ ,}
\end{eqnarray*}%
where $P_{\ell m}(.)$ denotes the associated Legendre functions. We can
hence consider random eigenfunctions of the form 
\begin{equation*}
T_{\ell }(x)=\sqrt{\dfrac{4\pi }{2\ell +1}}\sum_{m=-\ell }^{\ell }a_{\ell
m}Y_{\ell m}(x),
\end{equation*}%
where the coefficients $\{a_{\ell m}\}$ are independent, safe for the condition $a_{\ell
	m}=(-1)^{m}\overline{a}_{\ell ,-m}$; for $m\neq 0$ they are
standard complex-valued Gaussian variables,  while $a_{\ell 0}$ is a standard
real-valued Gaussian variable. The random fields $\{T_{\ell }(x),x\in \mathbb{S}^{2}\}
$ are Gaussian and isotropic, namely the probability laws of $T_{\ell
}(\cdot )$ and $T_{\ell }(g\cdot )$ are the same for any rotation $g\in
SO(3).$ 
Also, we have that 
\begin{eqnarray*}
\mathbb{E}[T_{\ell }(x)] &=&0,\mbox{ and }\mathbb{E}[T_{\ell }(x)^{2}]=1%
\text{ ,} \\
\mathbb{E}[T_{\ell }(x)T_{\ell }(y)] &=&P_{\ell }(\cos d(x,y)),
\end{eqnarray*}%
where $P_{\ell }$ are the Legendre polynomials and $d(x,y)$ is the spherical
geodesic distance between $x$ and $y,$ i.e. 
\begin{equation*}
d(x,y)=\arccos (\langle x,y\rangle ).
\end{equation*}%
The analysis of random eigenfunctions on the sphere or on other compact
manifolds (such as the torus) has been recently considered in many papers, due
to strong motivations arising from Cosmology and Quantum Mechanics, see
i.e., \cite{M e Peccati}, \cite{13}, \cite{21}, \cite{RW} and \cite{RWY}. Many papers have focussed on the geometry of the $z$%
-excursion sets, which are defined for $z\in \mathbb{R}$ as 
\begin{equation*}
A_{z}(\ell ):=A_{z}(T_{\ell },\mathbb{S}^{2}):=\{x\in \mathbb{S}^{2}:T_{\ell
}(x)>z\},
\end{equation*}%
see for instance \cite{M e W 2011}, \cite{M e W 2012}, \cite{M e W 2011bis}, 
\cite{mau}. More precisely, a natural tool to characterize the geometry of $%
\left\{ A_{z}(T_{\ell },\mathbb{S}^{2})\right\} $ is provided by the so-called
Lipschtz-Killing curvatures (see i.e., \cite{Adler e Taylor} p. 175), which in the
2-dimensional case correspond to the area of $A_{z}(\ell )$ (which we shall
write as $\mathcal{L}_{2}(A_{z}(T_{\ell },\mathbb{S}^{2})),$ (half) the
boundary length $\partial A_{z}(\ell )$ (i.e., the length of level curves $%
T_{\ell }^{-1}(z),$ written $\mathcal{L}_{1}(A_{z}(T_{\ell },\mathbb{S}%
^{2})))$, and their Euler-Poincaré characteristic, i.e., the difference
between the number of connected regions and the number of ``holes" (written $%
\mathcal{L}_{0}(A_{z}(T_{\ell },\mathbb{S}^{2}))$).

In order to characterize the stochastic properties of these functionals, the
first step clearly is the evaluation of their expected values. This goal can
be achieved by means of the celebrated Gaussian Kinematic Formula (see \cite%
{Adler e Taylor} chapter 13, Theorem 13.2.1), which yields, respectively, 
\begin{equation*}
\mathbb{E}[\mathcal{L}_{0}(A_{z}(T_{\ell },\mathbb{S}^{2}))]=2\{1-\Phi (z)\}+%
\frac{\lambda _{\ell }}{2}\frac{ze^{-z^{2}/2}}{\sqrt{(2\pi )^{3}}}4\pi ,
\end{equation*}%
for the Euler-Poincaré characteristic, 
\begin{equation*}
\mathbb{E}[\mathcal{L}_{1}(A_{z}(T_{\ell },\mathbb{S}^{2}))]=\pi \times 
\sqrt{\frac{\lambda _{\ell }}{2}}e^{-z^{2}/2},
\end{equation*}%
for (half) the boundary length, and 
\begin{equation*}
\mathbb{E}[\mathcal{L}_{2}(A_{z}(T_{\ell },\mathbb{S}^{2}))]=4\pi \times
\{1-\Phi (z)\},
\end{equation*}%
for the excursion area; note that $\frac{\lambda _{\ell }}{2}=\frac{\ell
(\ell +1)}{2}=P_{\ell }^{\prime }(1)$ represents the derivative of the
covariance function at the origin. Actually, in the Gaussian Kinematic Formula, the LKC are computed on the entire manifold and it depends only on metric properties, i.e., if the metric is scaled by $\lambda$, the LKC scales by $\lambda^k$. Hence, here, $\frac{\lambda_\ell}{2} $ is this scaling factor.  \newline

\bigskip The next step in the investigation of the random properties for
these functionals is the derivation of their variances and hence their limiting distributions. A crucial step to achieve these
results is to note that all these statistics can be written as nonlinear
functionals of the random fields itself and their spatial derivatives. For
instance, the excursion area can be expressed by 
\begin{equation*}
S_{\ell }(z)=\int_{\mathbb{S}^{2}}1_{\{z,+\infty \}}(T_{\ell }(x))\,dx,
\end{equation*}%
where $1_{A}(\cdot )$ is, as usual, the indicator function of the set $%
A,$ which takes value one if the condition in the argument is satified, zero
otherwise; likewise, using a Kac-Rice argument (see \cite{Azais e ws} Chapter 6, \cite%
{Adler e Taylor} Chapter 11) the length of level curves can be written as 
\begin{equation*}
\mathcal{L}_{\ell }(z)=\int_{\mathbb{S}^{2}}\delta _{z}(T_{\ell
}(x))||\nabla T_{\ell }(x)||\,dx,
\end{equation*}%
and a related expression can be given for the Euler-Poincaré characteristic
(see \cite{C M }). Starting from these expressions, it is possible to compute
explicitly the expansion of Lipschitz-Killing curvatures into the
orthonormal system generated by Hermite polynomials; for instance, in the
case of the excursion area it can be readily shown that (see \cite{M e W
2011}, \cite{M e W 2012}, \cite{M e Mau 2015})
\begin{equation*}
S_{\ell }(z)=\sum_{q=0}^{\infty }\dfrac{J_{q}(z)}{q!}\int_{\mathbb{S}%
^{2}}H_{q}(T_{\ell }(x))\,dx\text{,}
\end{equation*}%
the equality holding in the $L^{2}(\Omega )$-sense; we recall here the
standard definition of Hermite polynomials, i.e., $H_{0}\equiv 1$ and, for $q\geq
1,$ (see for instance \cite{Peccati Nourdin}) 
\begin{equation*}
H_{q}(x)=(-1)^{q}e^{\frac{x^{2}}{2}}\frac{d^{q}}{dx^{q}}e^{-\frac{x^{2}}{2}}%
\mbox{,  }x\in \mathbb{R}.
\end{equation*}%
The coefficients $\{J_{q}(\cdot )\}$ have the analytic expressions $%
J_{0}(z)=\Phi (z),$ $J_{1}(z)=-\phi (z),$ $J_{2}(z)=-z\phi (z),$ $%
J_{3}(z)=(1-z^{2})\phi (z)$ and in general 
\begin{equation*}
J_{q}(z)=-H_{q-1}(z)\phi (z),\text{ }q=1,2,3...
\end{equation*}%
where $\phi (\cdot )$ and $\Phi (\cdot )$ are the density function and the
distribution function of a standard Gaussian variable (\cite{M e W 2011}, 
\cite{M e W 2012}). As in \cite{M e W 2012}, we denote 
\begin{equation}
h_{\ell ,q}:=\int_{\mathbb{S}^{2}}H_{q}(T_{\ell }(x))\,dx\mbox{ }\mbox{ }%
q=1,2,\dots ,  \label{lug1}
\end{equation}%
and we can hence write 
\begin{equation}
S_{\ell }(z)=\sum_{q=0}^{\infty }\dfrac{J_{q}(z)}{q!}h_{\ell ;q}\mbox{ in }%
L^{2}(\Omega ).  \label{lug2}
\end{equation}%
It can be readily verified that the term corresponding to $q=1$ in (\ref%
{lug1}), (\ref{lug2}) are identically equal to zero for every $\ell \geq 1$;
indeed we have that%
\begin{equation}
h_{\ell,1 }:=\int_{\mathbb{S}^{2}}\sqrt{\dfrac{4\pi }{2\ell +1}}%
\sum_{m=-\ell }^{\ell }a_{\ell m}Y_{\ell m}(x)dx=\sqrt{\dfrac{4\pi }{2\ell +1%
}}\sum_{m=-\ell }^{\ell }a_{\ell m}\int_{\mathbb{S}^{2}}Y_{\ell m}(x)dx=0.
\label{lug3}
\end{equation}%
The crucial step in \cite{M e W 2011}, \cite{M e Mau 2015} is then to show
that a single term, corresponding to $q=2,$ has asymptotically (in the
high-energy regime $\ell \rightarrow \infty $) a dominating role in the
expansion, i.e.,%
\begin{equation*}
\Var(S_{\ell })=\left\{ \frac{J_{2}(z)}{2}\right\} ^{2}\Var(h_{\ell
;2})+o(\Var(S_{\ell })),\text{ as }\ell \rightarrow \infty \text{ ,}
\end{equation*}%
so that both the asymptotic variance and the Central Limit Theorem can be
established by simply considering the behaviour of this single term. Similar
expansions can be derived for the boundary length and the Euler-Poincaré
characteristic (see \cite{C M }, \cite{MRW2017}, \cite{W}, \cite{BW2016}), thus
leading to the following asymptotic expressions for the variances (see %
\cite{C M }, \cite{M e W 2012}, \cite{M e Mau 2015}, \cite{mau}):%
	\begin{equation*}
\begin{split}
&\lim_{\ell \rightarrow \infty} \ell^{-3} \Var(\mathcal{L}_{0}(A_{z}(T_{\ell },\mathbb{S}^{2})))=\dfrac{1}{4}%
(H_{3}(z)+H_{2}^{\prime }(z))^{2}\phi (z)^{2},\\
&
\lim_{\ell \rightarrow \infty}\ell^{-1}\mbox{Var}(\mathcal{L}_{1}(A_{z}(T_{\ell },\mathbb{S}^{2})))=\dfrac{\pi^3}{2} (H_{2}(z)+H_{1}^{\prime }(z))^2\phi (z)^{2}
\\&
\lim_{\ell \rightarrow \infty}\ell \mbox{Var}\lbrack \mathcal{L}_{2}(A_{z}(T_{\ell },\mathbb{S}^{2}))]= 4\pi^2(H_{1}(z)+H_{0}^{\prime }(z))^{2}\phi (z)^{2}.
\end{split}
\end{equation*}
See also \cite{11}, \cite{35}, \cite{M e Mau 2015}, \cite{W}, \cite{C}, \cite{P-R}, \cite{N-P-R}, \cite{MPRW}, \cite{DNPR} for related results on the
torus and on the plane, and \cite{MRW2017}, \cite{JW}, \cite{C M }, \cite{M e
Mau 2015}, \cite{M e W 2011}, \cite{RW}, \cite{RWY} for other works concerning the geometry of
random eigenfunctions on compact manifolds.
A common features of all these statistics is the disappearance of the
leading term at the zero level $z=0$ (the so-called Berry's cancellation
phenomenon, investigated in \cite{W}, \cite{M e W 2011bis}, \cite{35}). In
the case of the excursion area, at $z=0$ all the odd-order chaoses become
relevant, and the Central Limit Theorem can be established as in \cite%
{MRW2017}, \cite{C M }, \cite{MPRW}. For other functionals (nodal length), at 
$z=0$ the fourth order chaos plays the role of the dominant term.
Along the same lines, it has been possible to establish Quantitative Central
Limit Theorems for the asymptotic fluctuations in the high-energy regime. To
report these results, we need to introduce some more notation. Recall that
the Wasserstein $d_{W}$ distance between random variables $Z,N$ is defined by
\begin{equation*}
d_{W}(Z,N):=\sup_{h\in Lip(1)}|\mathbb{E}[h(Z)]-\mathbb{E}[h(N)]|
\end{equation*}%
where $Lip(1)$ denotes the set of Lipschitz functions with bounding constant
equal to 1. It should be noted that the functionals $\left\{ h_{\ell
;q}\right\} $ belong to the so-called Wiener chaoses of order $q,$ defined
as the space spanned by linear combinations of Hermite polynomials of order $%
q;$ as such, they belong to the domain of application for the so-called
Stein-Malliavin method, leading to very neat characterizations for
Quantitative Central Limit Theorems (see i.e., \cite{NP2009}, \cite{Peccati
Nourdin}). More precisely, we have that (Theorem 5.2.6, p. 99 \cite{Peccati
Nourdin}) 
\begin{equation}
d_{W}\bigg(\dfrac{h_{\ell ;q}}{\sqrt{\Var(h_{\ell ;q})}},Z \bigg)\leq 2%
\sqrt{\dfrac{q-1}{3q}\bigg(\dfrac{cum_{4}(h_{\ell ;q})}{\Var^{2}(h_{\ell ;q})}%
\bigg)},  \label{N-P}
\end{equation}%
where $Z \sim \mathcal{N}(0,1)$ and $cum_{4}(Y):=\mathbb{E}Y^{4}-3\mathbb{E}Y^2$ denotes the fourth-order cumulant of a random variable
Y. In words, this means that in these circumstances to prove a Quantitative
Central Limit Theorem for standardized sequences it is enough to show that
their fourth-order moment goes to 3. This approach was used to establish
Quantitative Central Limit Theorems in \cite{M e W 2011},\cite{M e Mau 2015}
(see also \cite{M e W 2012}, \cite{MRW2017}, \cite{C M }, \cite{MPRW}), i.e.,
for $z\neq 0$, 
\begin{equation}\label{tlc-sfera}
d_{W}\bigg(\dfrac{S_{\ell }(z)-\mathbb{E}[S_{\ell }(z)]}{\sqrt{\Var(S_{\ell
}(z))}},Z\bigg)=O\big(\ell ^{-1/2}\big),
\end{equation}%
as $\ell \rightarrow \infty ,$ entailing as a Corollary that 
\begin{equation*}
\dfrac{S_{\ell }(z)-\mathbb{E}[S_{\ell }(z)]}{\sqrt{\Var(S_{\ell }(z))}}%
\rightarrow _{d}Z,\mbox{ }z\neq 0
\end{equation*}%
$d$ denoting the convergence in distribution.

\subsection{Main Result}
In this paper we extend and generalize some of the previous results,
considering the case of the excursion area evaluated on a spherical cap
rather than the full sphere. More precisely, we shall focus on a symmetric
spherical cap $B$ of radius $r<\pi ,$ which we take without loss of
generality to be centred around the North Pole $N=(0,0),$ i.e.,%
\begin{equation}\label{Bdef}
B=\left\{ x\in \mathbb{S}^{2}:0\leq \theta _{x}\leq r,\text{ }0\leq \varphi
_{x}\leq 2\pi \right\} \text{.}
\end{equation}%
We shall then consider the excursion set 
\begin{equation*}
A_{z}(T_{\ell },B)=\{x\in B:T_{\ell }(x)>z\},
\end{equation*}%
and in particular the excursion area 
\begin{equation*}
S_{\ell }(B,z):=\int_{B}1_{\{T_{\ell }(x)>z\}}(T_{\ell }(x))\,dx\text{.}
\end{equation*}%
Note that the mean value of the excursion area is simply given by $$\mathbb{E}[S_\ell(B,z)]= \int_B \mathbb{E}[1_{\{T_\ell(x)>z\}}(T_\ell(x))] \,dx=(1-\phi(z)) m(B),$$
where $m(B)$ is the measure of $B$.\\
Our main result is a Quantitative Central Limit Theorem of the form
\begin{thm}
\label{tlc} For every $z\neq 0$, as $\ell \rightarrow \infty ,$ we have that 
\begin{equation*}
d_{W}\bigg(\frac{S_{\ell }(B,z)-\mathbb{E}[S_{\ell }(B,z)]}{\sqrt{\Var(S_{\ell }(B,z))}},Z\bigg)=O%
\bigg(\frac{1}{\sqrt{\ell }}\bigg),
\end{equation*}
where $Z \sim \mathcal{N}(0,1)$.
\end{thm}
The main steps in the proof of this result are described in the next
section. We anticipate that our main ideas are broadly similar to those
previously exploited in the related literature: namely, we compute the $L^{2}$-expansion into Hermite polynomials and we show that the second
order term is the dominating one. Along these similarities, we stress
however that there exist as well very important differences, which we list
below as follows:
\begin{itemize}
\item While the first-order chaos term is identically zero in the case of
the full sphere (see (\ref{lug3})), this result does no longer hold on
subdomains and a careful analysis is needed to show that the corresponding
terms are of lower stochastic order. Here we shall also require the
properties of a smooth approximation for the indicator function of the
spherical cap, whose construction is of some independent interest (Section \ref{mol}).

\item The second-order chaos term is still the leading one in the $L^{2}$
expansion, and it decays to zero with the same rate $\ell ^{-1}$ as in the
full spherical case. However, the normalizing constants are different, and
they can be given a natural interpretation as the relative area of the
region under consideration.

\item It is still possible to show that a (Quantitative) Central Limit
Theorem holds. However the proof is entirely different from the one
exploited in the case of the full sphere, and indeed much more challenging.
In fact, due to Parseval's identity, in the case of the full sphere the
second-order chaos boils down to a simple sum of independent and identically
distributed random variables, so that the Central Limit Theorem, even in its
Quantitative version, is almost immediate. Here, on the contrary, these
identities no longer hold, and it thus becomes necessary to exploit the full
power of Stein-Malliavin results (see \cite{Peccati Nourdin} and \cite{NP2009}) by means of a careful computation of
fourth-order cumulants. In particular, the latter result requires the
investigation of complex cross-sums of so-called Clebsch-Gordan coefficients
(see \cite{quantum theory}, \cite{M e Peccati}), which arise from integrals
of multiple products of spherical harmonics. Finally,
it is remarkable that the asymptotic rate of convergence in the Quantiative
Central Limit Theorem turns out to be identical to the full spherical case.

\item It remains true that the leading term in the variance expansion
vanishes in the ``nodal" case $z=0,$ i.e., some form of the Berry's
cancellation phenomenon (see \cite{Berry1977}, \cite{mau}, \cite{W}) applies
to subdomains of the sphere as well.
\end{itemize}

\subsection{Plan of the paper}
In Section 2 we briefly explain the ideas of the proof of the main result, while Section 3 discusses the construction of a smooth
approximation to the indicator function and its asymptotic properties. The
proof of the Central Limit Theorem is given in Section 4, where the
asymptotic behavior of the chaotic components of the excursion area are investigated.
Further technical computations are collected in the Appendix.

In the sequel, given any two positive sequence $a_{n},b_{n}$, we shall write 
$a_{n}\sim b_{n}$ if $\lim_{n\rightarrow \infty }a_{n}/b_{n}=1.$ Also we
shall write $a_{n}\ll b_{n}$ or $a_{n}=O(b_{n})$ when the sequence $%
a_{n}/b_{n}$ is asymptotically bounded.

\subsection{Acknowledgements}
This topic was suggested by Domenico Marinucci and the author would like to
thank him for this and for all the suggestions, the discussions and the
useful comments. I would like to thank the departement of Mathematics of the
University of Rome \textquotedblleft Tor Vergata", where part of the
research was done, for the warm ospitality. Finally, thanks to Igor Wigman and Maurizia
Rossi for many insights and useful discussions. 

\section{On the proof of the main result}
From now on $B$ will denote the spherical cap defined in (\ref{Bdef}).
As mentioned earlier, in order to study the excursion area, we start by writing it as the
functional 
\begin{equation*}
S_{\ell }(B,z)=\int_{B}1_{\{T_{\ell }(x)>z\}}(T_{\ell }(x))dx
\end{equation*}%
and then, exploiting the $L^{2}$-expansion into Wiener chaoses, we have
\begin{equation}\label{jen}
1_{\{T_\ell(x) \leq z\}}(T_\ell(x))=\sum_{q=0}^{\infty} \frac{J_q(z)}{q!}
H_q(T_\ell(x)),
\end{equation}
meaning that 
\begin{equation*}
\lim_{Q \rightarrow \infty} E[|\sum_{q=0}^{Q} \dfrac{J_q(z)}{q!}
H_q(T_\ell(x))-1_{\{T_\ell(x)\leq z \}}(T_\ell(x))|^2]=0.
\end{equation*}
Because of the linearity of the integral and Jensen inequality, one has 
\begin{equation*}
\int_{B}1_{\{T_{\ell }(x)\geq z\}}(T_\ell(x))\,dx=\lim_{Q\rightarrow \infty
}\sum_{q=0}^{Q}\dfrac{J_{q}(z)}{q!}\int_{B}H_{q}(T_{\ell
}(x))\,dx=\lim_{Q\rightarrow \infty }\sum_{q=0}^{Q}\dfrac{J_{q}(z)}{q!}%
h_{\ell ;q}(B)
\end{equation*}%
where
\begin{equation*}
h_{\ell ;q}(B):=\int_{B}H_{q}(T_{\ell }(x))\,dx;
\end{equation*}%
Indeed, 
$$\mathbb{E} \bigg[ \bigg| \sum_{q=0}^{Q} \frac{J_q(z)}{q!} \int_B H_q(T_\ell(x)) \,dx-\int_B 1_{\{T_\ell(x)\leq z\}}(T_\ell(x)) \bigg|^2\bigg] \leq \mathbb{E} \bigg[ \int_B \bigg| \sum_{q=0}^{Q} \frac{J_q(z)}{q!} H_q(T_\ell(x)) \,dx- 1_{\{T_\ell(x)\leq z\}}(T_\ell(x)) \bigg|^2 \,dx \bigg] $$
which goes to zero for (\ref{jen}).
We can hence write
\begin{equation}\label{serie}
\begin{split}
\int_{B}1_{\{T_{\ell }(x)>z\}}(T_\ell(x))\,dx& =\int_{B}\left\{ 1-\Phi (z)\right\}
dx+\int_{B}\phi (z)H_{1}(T_{\ell }(x))dx+\int_{B}z\phi (z)\frac{1}{2}%
H_{2}(T_{\ell }(x))dx \\
& +\int_{B}\sum_{q=3}^{\infty }\dfrac{J_{q}(z)}{q!}H_{q}(T_{\ell }(x))\,dx,
\end{split}
\end{equation}%
in the $L^{2}(\Omega )-$convergence sense. The same holds for the variance
thanks to the continuity of the norm. Indeed 
\begin{equation}
\begin{split}
\Var\bigg(& \int_{B}1_{\{T_{\ell }(x)>z\}}(T_{\ell }(x))\,dx\bigg)=E\bigg[%
\bigg(\int_{B}1_{\{T_{\ell }(x)>z\}}(T_{\ell }(x))\,dx\bigg)^{2}\bigg]= \\
& =\left\langle \int_{B}1_{\{T_{\ell }(x)>z\}}(T_{\ell
}(x))\,dx,\int_{B}1_{\{T_{\ell }(x)>z\}}(T_{\ell }(x))\,dx\right\rangle
_{L^{2}(\Omega )}= \\
& =\lim_{Q\rightarrow \infty }\left\langle \sum_{q=0}^{Q}\dfrac{J_{q}(z)}{q!}%
\int_{B}H_{q}(T_{\ell }(x))\,dx,\sum_{q=0}^{Q}\dfrac{J_{q}(z)}{q!}%
\int_{B}H_{q}(T_{\ell }(x))\,dx\right\rangle.
\end{split}%
\end{equation}
Hence the following expansion holds in $L^{2}(\Omega )$ sense: 
\begin{equation}
\begin{split}
\mbox{Var}\bigg(\int_{B}1_{\{T_{\ell }(x)>z\}}(T_\ell(x))dx\bigg)& =0+\phi (z)^{2}%
\mbox{Var}\bigg(\int_{B}H_{1}(T_{\ell }(x))dx\bigg)+\dfrac{z^{2}\phi (z)^{2}%
}{4}\mbox{Var}\bigg(\int_{B}H_{2}(T_{\ell }(x))dx\bigg) \\
& +\mbox{Var}\bigg(\int_{B}\sum_{q=3}^{\infty }\dfrac{J_{q}(z)}{q!}%
H_{q}(T_{\ell }(x))dx\bigg).
\end{split}
\label{serie-var}
\end{equation}%
The Quantitative Central Limit Theorem is established by the analysis of the
asymptotic behaviour for each of these terms; here below we give a summary
of the results we obtained for the singular components. \\\\
In the sequel, we
shall need a continuous differentiable function, which we denote as $1_{B,\varepsilon }(x)$, for $\eps >0$, converging to the indicator function in $L^1(\mathbb{S}^2)$, as $\eps \rightarrow 0.$
\begin{oss}
	Along the framework we will refer to a particular smooth function, constructed in Section \ref{mol}. However, we would like to stress that the specific choice of the mollifier function is not important; more precisely, the fundamental issue is the behaviour of its Fourier coefficients: indeed, we need them to go to zero quite $``$fast$"$, in order to exchange integrals and series and to work with absolute convergent series. In Section \ref{mol} we give just an example of such a possible function.
\end{oss}
Here below, we summarize the conditions which $\eps$ has to satisfy in order to prove Theorem \ref{tlc}. Note that we take $\eps:=\eps_{\ell}$ as a sequence depending on $\ell$; nevertheless, we drop the $\ell$ whenever possible for notational simplicity.
\begin{ass}\label{cond-eps}
	Let us consider $1_{B,\varepsilon }(x)$ the smooth function constructed in Section \ref{mol}, then $\eps=\eps_{\ell}$ is such that
	\begin{equation}\label{cond2} 
	\ell^{\frac{3-M}{2M+1}}<\eps_\ell <\ell^{-\frac{1}{3}}.
	\end{equation}
	The parameter $M$ will be fixed below (see Section \ref{mol}).
\end{ass}
The condition on the left ensures the convergence of the Fourier coefficients $b_{\ell;\eps}$ of the Fourier expansion of $1_{B;\eps}(x)$, namely
\begin{equation}\label{fourier}
1_{B;\eps}(x)=\sum_\ell b_{\ell;\eps} Y_{\ell 0}(x),
\end{equation} and the absolutely convergence of (\ref{fourier}) (see (\ref{perche cond1})). Whereas, the bound on the right hand side in (\ref{cond2}) makes the first chaotic component smaller than the second one (see the proof of Proposition \ref{cor1chaos}).
\begin{ex}
	If we set for instance $\eps=\eps_{\ell}=\dfrac{1}{\ell^\alpha}$, with $\alpha>0, \alpha \in \mathbb{R}$, the lower bound in (\ref{cond2}) implies
	$$ \alpha<\dfrac{M-3}{2M+1} $$
	and the upper bound
	$$ \alpha>\dfrac{1}{3};$$
	hence, a sequence $\eps_{\ell}$ satisfying Assumption \ref{cond-eps} exists  taking $M>10$.
\end{ex}

\vspace{0.5cm}
Since $H_1(T_\ell(x))=T_\ell(x),$ we obtain, for the first chaotic component, the following proposition.

\begin{prop}
\label{cor1chaos} Let $B$ the spherical cap defined in (\ref{Bdef}), under the Assumption \ref{cond-eps}, the variance of the first chaotic component of (\ref{serie-var}) is 
\begin{equation*}
\Var \bigg(\int_{\mathbb{S}^{2}}1_{B}(x)T_{\ell }(x)\,dx\bigg)=o\bigg(\dfrac{1}{\ell }%
\bigg)
\end{equation*}%
as $\ell \rightarrow \infty .$
\end{prop}
To establish this result, we write the variance as 
\begin{equation}\label{key1}
\begin{split}
\Var\bigg(\int_{\mathbb{S}^{2}}1_{B}(x)T_{\ell }(x)\,dx\bigg)& =\Var\bigg(%
\int_{\mathbb{S}^{2}}(1_{B}(x)-1_{B,\varepsilon }(x))T_{\ell }(x)\bigg)+\Var\bigg(%
\int_{\mathbb{S}^{2}}1_{B,\varepsilon }(x)T_{\ell }(x)\,dx\bigg)+ \\
& +\mathbb{E}\bigg[\int_{\mathbb{S}^{2}}(1_{B,\varepsilon }(x)-1_{B}(x))1_{B;\varepsilon
}(y)T_{\ell }(x)T_{\ell }(y)\bigg].
\end{split}%
\end{equation}%
The second integral in (\ref{key1}) will be computed to be 
\begin{equation*}
\Var\bigg(\int_{\mathbb{S}^{2}}1_{B,\varepsilon }(x)T_{\ell }(x)\,dx\bigg)=\frac{4\pi 
}{2\ell +1}b_{\ell ,\varepsilon }^{2};
\end{equation*}%
where, we have already said, $b_{\ell ,\varepsilon }$ are the Fourier coefficients of $%
1_{B,\varepsilon }(x)$, given in Theorem \ref{mollifier}. The former
and the latter terms in (\ref{key1}) will be proved to be of order $\frac{1}{\sqrt{\ell}}\varepsilon ^{3/2}$, for $\eps \rightarrow 0$. Hence, the right hand side of Assumption $\ref{cond-eps}$ implies
the thesis of Proposition \ref{cor1chaos}.\\

As far as the second chaotic component is concerned, the following proposition will be proved.

\begin{prop}\label{2chaotic} Let us consider $1_{B,\eps}(x)$, for $\eps>0$, the continuous function constructed in Section \ref{mol}, converging to $1_{B}(x)$ in $L^1(\mathbb{S}^2)$, with $\eps$ satisfying Assumption \ref{cond-eps}. It can be proved (see Lemma \ref{lemma1}) that
	\begin{equation}\label{key3}
	\Var\bigg( \int_{\mathbb{S}^2} 1_{B,\eps}(x) H_2(T_\ell(x))\,dx\bigg)=8\pi \sum_{\ell _{1}}b_{\ell
		_{1};\varepsilon }^{2}\dfrac{1}{2\ell _{1}+1}\bigg(C_{\ell 0\ell 0}^{\ell
		_{1}0}\bigg)^{2},
	\end{equation}
	where $\{C_{\ell 0\ell 0}^{\ell _{1}0}\}$ are
	the Clebsch-Gordan coefficients (\cite{quantum theory}, Chapter 8 or Appendix \ref{C-G Appendix}).
	Then, the variance of the second chaotic projection of the excursion area in (\ref{serie-var}) is
	\begin{equation}
	\Var\bigg(\int_{B}H_{2}(T_{\ell }(x))\,dx\bigg)=8\pi \sum_{\ell _{1}}b_{\ell
		_{1};\varepsilon }^{2}\dfrac{1}{2\ell _{1}+1}\bigg(C_{\ell 0\ell 0}^{\ell
		_{1}0}\bigg)^{2}+o\bigg(\frac{1}{\ell }\bigg),  \label{variance2chaos}
	\end{equation}%
	as $\ell \rightarrow \infty$, where the bound is uniformly in $\eps$.
\end{prop}

Note that, to prove Proposition \ref{2chaotic}, it is sufficient that $\eps$ satisfies ($\ref{cond2}$).

\begin{oss}\label{norma} It is easy to see with Parseval's identity that
	$\sum_\ell b_{\ell;\varepsilon}^2=||1_{B;\varepsilon} ||_{L^2(\mathbb{S}^2)}^2 \leq
	m(\mathbb{S}^2)= 4\pi$.
\end{oss}

\begin{oss}It is interesting to compare the results in Proposition \ref{2chaotic} with the one in the case of the full sphere. Then, let us consider $B=\mathbb{S}^2$, i.e. $1_B(\cdot)=1_{\mathbb{S}^2}(%
	\cdot);$ in this case the approximating function $1_{B;\varepsilon}(\cdot)$ is not necessary. Indeed, the only
	term of the Fourier expansion of the indicator function $%
	1_{\mathbb{S}^2}(\cdot)$ is $\ell_1=0,$ moreover, for (\ref{1 pag 248})
	\begin{equation*}
	C_{\ell0\ell0}^{00}= \dfrac{1}{\sqrt{2\ell+1}},
	\end{equation*}
	\begin{equation*}
	\mbox{ and }\mbox{ }\mbox{ }\mbox{ } b_0=2\pi \int_{0}^{\pi} \dfrac{1}{\sqrt{4\pi}} \sin \theta d\theta=\dfrac{%
		4\pi}{\sqrt{4\pi}}=\sqrt{4\pi},
	\end{equation*}
	so that
	\begin{equation*}
	8\pi \sum_{\ell_1=0}^{\infty} b_{\ell_1;\varepsilon}^2 \dfrac{1}{2\ell_1+1}%
	(C_{\ell 0\ell 0}^{\ell_1 0})^2=8\pi b_0^2 (C_{\ell 0 \ell 0}^{00})^2= 32\pi^2
	\dfrac{1}{2\ell+1},
	\end{equation*}
	hence, the variance is
	\begin{equation*}
	\Var\bigg(\int_B H_2(T_\ell(x)) \,dx \bigg)= 32\pi^2 \dfrac{1}{2\ell+1}
	\sim 16 \pi^2 \dfrac{1}{\ell},
	\end{equation*}
	that is exactly the value obtained in \cite{M e W 2011}, Proposition 2.1.
\end{oss}

The idea of the proof of Proposition \ref{2chaotic} is similar to the one given in Proposition \ref{cor1chaos}. More precisely,
we write 
\begin{equation}\label{un giorno di pioggia}
\begin{split}
\Var\bigg(\int_{\mathbb{S}^{2}}H_{2}(T_{\ell }(x))\,dx\bigg)& =\Var\bigg(%
\int_{\mathbb{S}^{2}}(1_{B}(x)-1_{B,\varepsilon }(x))H_{2}(T_{\ell }(x))\,dx\bigg)\\&+\Var%
\bigg(\int_{\mathbb{S}^{2}}1_{B,\varepsilon }(x)H_{2}(T_{\ell }(x))\,dx\bigg) \\
& +2\mathbb{E}\bigg[\int_{\mathbb{S}^{2}\times \mathbb{S}^{2}}1_{B,\varepsilon }(x)\bigg(%
1_{B}(y)-1_{B,\varepsilon }(y)\bigg)H_{2}(T_{\ell }(x))H_{2}(T_{\ell
}(y))\,dxdy\bigg].
\end{split}%
\end{equation}%
The first integral in (\ref{un giorno di pioggia}) can be shown to be smaller than $\dfrac{Const\cdot
	\varepsilon}{2\ell +1}$ (see below (\ref{key2})), which is a $o\bigg(\dfrac{1}{%
	\ell }\bigg)$ since $\varepsilon \rightarrow 0$; the same bound holds for
the third integral in (\ref{un giorno di pioggia}), in view of the Cauchy-Schwarz inequality. Whereas, for the second
integral in (\ref{un giorno di pioggia}), the validity of (\ref{key3}) can be proved (see Lemma \ref{lemma1}) %
and it is seen that its asymptotic order is $\dfrac{1}{\ell }$ (see Lemma \ref{lemma2}). The
proof of (\ref{key3}) is based on manipulations of spherical
harmonics and their integrals. More precisely, we make use of the addition
formula (see for example \cite{M e Peccati}, eq. (3.42) p. 66):
\begin{equation}
\sum_{m=-\ell }^{\ell }\overline{Y}_{\ell m}(x)Y_{\ell m}(y)=\dfrac{2\ell +1%
}{4\pi }P_{\ell }(\langle x,y\rangle )\text{ ;}  \label{add formula}
\end{equation}%
moreover, recalling that 
\begin{equation}
Y_{\ell 0}(\theta ,\varphi )=\sqrt{\dfrac{2\ell +1}{4\pi }}P_{\ell }(\cos
\theta ),
\end{equation}%
using the expansion%
\begin{equation}
1_{B;\varepsilon }(x)=\sum_{\ell =0}^{\infty }b_{\ell ;\varepsilon }Y_{\ell
0}(x)  \label{fn caratteristica}
\end{equation}%
and replacing these formulae in the left hand side in (\ref{key3}), we obtain
the so-called Gaunt integral of spherical harmonics (\cite{M e Peccati} eq. (3.64) p. 81)
which can be computed by the following relation: 
\begin{equation}\label{ullalla}
\int_{\mathbb{S}^{2}}Y_{\ell _{1}m_{1}}(x)Y_{\ell _{2}m_{2}}(x)\overline{Y}_{\ell
_{3}m_{3}}(x)d\sigma (x)=\sqrt{\dfrac{(2\ell _{1}+1)(2\ell _{2}+1)}{4\pi
(2\ell _{3}+1)}}C_{\ell _{1}m_{1}\ell _{2}m_{2}}^{\ell _{3}m_{3}}C_{\ell
_{1}0\ell _{2}0}^{\ell _{3}0}, 
\end{equation}%
for all $\ell _{1},\ell _{2},\ell _{3}.$ Finally, the proof is completed by
a careful analysis of properties for the Clebsch-Gordan coefficients, most
of which are reported in the Appendix.\\

The next important step in our argument is to establish the Quantitative
Central Limit Theorem. This argument requires two steps; first we need to
show that the variance of all higher-order chaoses for $q\geq 3$ is of
smaller order; this can be done quite simply by some rather easy
majorizations, which allow to show that all these terms are of order $o\bigg(%
\dfrac{1}{\ell }\bigg)$. On the other hand, since the second term is the
leading component, we need to compute its fourth-order cumulant to be able
to apply Theorem 5.2.6 in \cite{Peccati Nourdin} and
hence to establish asymptotic Gaussianity.

However, the difficulty to handle computations with the indicator function leads us to consider the fourth cumulant of 
\begin{equation}\label{h2star}
h_{\ell,2}^*(B):=\int_{\mathbb{S}^2} 1_{B,\eps}(x) H_2(T_\ell(x))\, dx;
\end{equation}
more precisely, we shall show that

\begin{prop}	\label{4 cumulant} Under the assumptions of Proposition \ref{2chaotic}, the fourth cumulant of (\ref{h2star}) satisfies
	\begin{equation*}
	cum_4(h_{\ell;2}^*(B))=O\bigg(\dfrac{1}{\ell^3}\bigg),
	\end{equation*}
	as $\ell \rightarrow \infty.$
\end{prop}
Our approach in Proposition \ref{4 cumulant} is different from the one used in related circumstances
by for instance \cite{M e W 2011bis}, \cite{CMW2016}, \cite{Rossi}; indeed these papers use an approximation of Legendre polynomials
known as Hilb's asymptotics (see \cite{M e W 2012}, \cite{M e W 2011bis}): however this approximation turned out not to be
efficient enough in the present framework. Hence, we need to exploit a
different argument, i.e., we
compute the exact values of the multiple integrals for spherical harmonics
by means of Gaunt integrals (\ref{ullalla}) (see \cite{M e Peccati}) and Clebsch-Gordan coefficients.\\

At this stage, thanks to Theorem 5.2.7 \cite{Peccati Nourdin}, the following bound holds 
\begin{equation*}
d_{W}\bigg(\dfrac{h_{\ell;2 }^*(B)}{\sqrt{\Var(h_{\ell ;2}^*(B))}},Z\bigg)\leq \sqrt{\dfrac{1}{6}\bigg(\dfrac{cum_{4}(h_{\ell;2 }^*(B))}{%
		\Var(h_{\ell;2 }^*(B))^{2}}\bigg)},
\end{equation*}%
then, CLT can be proved for $h_{\ell,2}^*(B)$. Now, in view of the fact that 
\begin{equation}\label{bound per w}
\mathbb{E}\left[ \int_{\mathbb{S}^{2}}1_{B,\varepsilon }(x)H_{2}(T_{\ell
}(x))\,dx-\int_{\mathbb{S}^{2}}1_{B}(x)H_{2}(T_{\ell }(x))\,dx\right] ^{2}=o \bigg(\dfrac{1}{\ell} \bigg)%
\text{ as }\varepsilon \rightarrow 0,
\end{equation}%
we prove Theorem \ref{tlc} exploiting the triangular inequality for the Wasserstein distance. Note that, the result of Theorem \ref{tlc} is the same obtained for the sphere in (\ref{tlc-sfera})  (see also \cite{M e Mau 2015}).


\section{Construction of a mollifier for the characteristic function}\label{mol}
This section can be considered of some independent interest; it describes a method to construct an approximation of the indicator function, i.e., it gives an explicit expression for the function $1_{B,\eps}$, already mentioned, converging to the indicator function $1_B{(\cdot)}$ in $L^1(\mathbb{S}^2)$.\\\\
For any fixed $M>0, M \in \mathbb{N}$, a general method to construct a function $%
\phi (\cdot )\in C^{M}$, can be given by the
B-splines approach (see \cite{M e Peccati}, p. 250), as follows. First of all, recall that the Bernstein polynomials
are defined as 
\begin{equation*}
B_{i}^{(n)}(t):=\binom{n}{i}t^{i}(1-t)^{n-i},
\end{equation*}%
where $t\in \lbrack 0,1],$ $i=0,\dots ,n$ and $n=1,2,\dots .$ Then, we can
define polynomials 
\begin{equation*}
q_{2k+1}(t):=\sum_{i=0}^{k}B_{i}^{(2k+1)}(t);
\end{equation*}%
one has that  $q_{2k+1}(0)=1$ and $q_{2k+1}(1)=0.$ Moreover, 
\begin{equation*}
q_{2k+1}^{(m)}(1)=q_{2k+1}^{(m)}(0)=0\mbox{ for }m=1,\dots ,k.
\end{equation*}%
Hence, let  $r \in (0,\pi)$ and $\theta \in [0,\pi)$, for any $\eps >0$ we set
\begin{equation*}
t:=\dfrac{\theta -(r-\varepsilon )}{r-(r-\varepsilon )} \mbox{ } \in [0,1]
\end{equation*}%
and define the function 
\begin{equation}
\phi _{r,\varepsilon }(\theta ):=%
\begin{cases}
1 & \mbox{ if }\theta \in \lbrack 0,r-\varepsilon ) \\ 
q_{2k+1}(t)=q_{2k+1}(\frac{\theta -r+\varepsilon }{\varepsilon }) & \mbox{
if }\theta \in \lbrack r-\varepsilon ,r] \\ 
0 & \mbox{ if }\theta \in \lbrack r,\pi )%
\end{cases}%
\end{equation}%
with $\theta \in (0,\pi )$.

\begin{figure}[h!]
\centering
\includegraphics[width=0.8\linewidth]{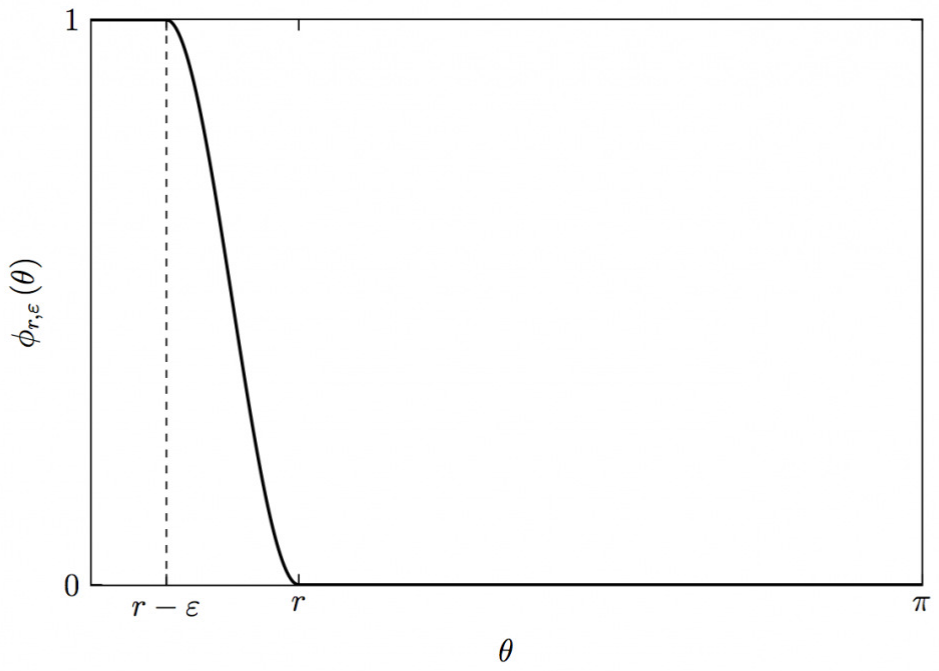} 	\caption{}
\label{fig:fig1.jpg}
\end{figure}

The function $\phi_{r,\varepsilon}(\theta)$ is a $2k+1$-degree polynomial, so $%
\phi_{r,\varepsilon} \in C^{M}$ for $M< k+1/2$ and $\phi_{r,\varepsilon}(
r-\varepsilon)=q_{2k+1}(0)=1$.

\begin{oss}\label{approx}
The indicator function $1_B(x), x \in \mathbb{S}^2$ can be written in spherical coordinates as $1_{B}(\theta, \varphi)$ with $\theta \in [0,\pi)$ and $\varphi \in [0,2\pi]$ but it only depends on the angle $\theta$, namely,
\begin{equation}
1_{B}(\theta,\varphi)=%
\begin{cases}
1 & \theta \leq r \\ 
0 & otherwise%
\end{cases}
=1_B(\theta).
\end{equation}
Defining $1_{B,\varepsilon}(\theta):=\phi_{r,\varepsilon}(\theta)$, it is easily to see that, as $\eps \rightarrow 0$, $1_{B,\eps}(\cdot) \rightarrow 1_{B}(\cdot)$ in $L^1(\mathbb{S}^2)$. In fact,
\begin{equation*}
\begin{split}
\int_{\mathbb{S}^2} |1_B(x)-1_{B,\varepsilon}(x)| dx&=2\pi \int_{0}^{\pi}
|1_B(\theta)-1_{B,\varepsilon}(\theta)| \sin \theta d\theta \\& \leq 2\pi \int_{
r-\varepsilon}^{ r} \bigg|q_{2k+1}\bigg(\dfrac{\theta-\cos r+\varepsilon}{%
\varepsilon}\bigg) \bigg| \sin \theta d \theta \leq 2\pi \varepsilon
\rightarrow 0,
\end{split}
\end{equation*}
as $\varepsilon \rightarrow 0.$
\end{oss}
Now we focus on the function $\phi_{r,\varepsilon}(\cdot)$. As denoted in \cite{Lang e Schwab}, we define $%
k_{r,\varepsilon}(\mu):=\phi_{r,\varepsilon}( \arccos \mu)$ with $\mu \in
[-1,1]$. Now recall that any function $u \in L^2(-1,1)$ can be
expanded in the $L^2(-1,1)$ convergent Fourier-Legendre series as 
\begin{equation*}
u=\sum_{\ell=0}^{\infty} u_\ell \frac{2\ell+1}{2}P_\ell 
\mbox{ }\mbox{ with } \mbox{ }
u_\ell=\int_{-1}^{1} u(x)P_\ell(x) \, dx
\end{equation*}
and hence
\begin{equation*}
b_\ell=2\pi u_\ell\sqrt{\dfrac{2\ell+1}{4\pi}}=\int_{-1}^{1} u(x) Y_\ell(x)
\, dx;
\end{equation*}

thus we can expand $k_{r,\eps}$ in such a series and its Fourier coefficients are 
\begin{equation*}
b_{\ell,\varepsilon}^r=\sqrt{\dfrac{2\ell+1}{4\pi}} \int_{-1}^{1}
k_{r,\varepsilon}(\mu)P_\ell(\mu) \,d\mu.
\end{equation*}
\begin{oss}\label{remark b0}
For $\ell=0$, it is easy to see that $	b_{0,\varepsilon}^r$ is bounded above and below by two positive constants. Actually, by definition,
	\begin{equation*}
	b_{0,\varepsilon}^r= \sqrt{\dfrac{1}{4\pi}} \int_{-1}^{1} k_{r,\varepsilon}(
	\theta) \, d\theta= \dfrac{1}{\sqrt{4\pi}} \int_{-1}^{1} \phi_{r,\varepsilon}(\arccos \theta) \,d \theta;
	\end{equation*}
	changing cordinates $\arccos \theta=x$, one has
	\begin{equation}
	\begin{split}
	b_{0,\varepsilon}^r&= \dfrac{1}{\sqrt{4\pi}} \int_{0}^{\pi}
	\phi_{r,\varepsilon} (x) \sin x \,dx \\
	&= \dfrac{1}{\sqrt{4\pi}} \int_{0}^{r-\varepsilon} \sin x\,dx + \dfrac{%
		1}{\sqrt{4\pi}} \int_{r-\varepsilon}^{r} q_{2k+1} (x) \sin x \,dx \\
	& \geq \dfrac{1}{\sqrt{4\pi}} \int_{0}^{r-\varepsilon} \sin x\,dx= 
	\dfrac{1}{\sqrt{4\pi}}(1-\arccos (r-\eps)) \geq \dfrac{1-r+\eps}{\sqrt{4\pi}}>\dfrac{1-r}{\sqrt{4\pi}}
	\end{split}%
	\end{equation}
	and since 
	\begin{equation*}
	|\phi_{r,\varepsilon}(\theta)| \leq 1,
	\end{equation*}
	it is immediate to conclude that $$ \dfrac{1-r}{\sqrt{4\pi}} \leq b_{0,\varepsilon}^r\leq \dfrac{1}{\sqrt{\pi}}.$$
\end{oss}

The main result of this section is given in the proposition below, which yields a bound for the Fourier coefficients $b_{\ell,\eps}^r$.
\begin{prop}
\label{coeff} For any fixed $M \in \mathbb{N}$ and $r \in (0,\pi)$, there exists a constant $K_{M,r}$
such that 
\begin{equation*}
|b_{\ell,\varepsilon}^r|\leq \min \bigg\{ b_{0,\eps}^r,\dfrac{K_{M,r}}{ \ell^{M-\frac{1}{2}}
\varepsilon^{2M+1} } \bigg\} .
\end{equation*}
\end{prop}
In order to prove Proposition \ref{coeff}, we get a bound for the $M-$derivative of $%
k_{r,\varepsilon}$. Since $k_{r,\varepsilon}(\mu)$ is a composite function, Faà di Bruno's formula implies:
\begin{equation}\label{faadibruno}
D^{M}(\phi_{r,\varepsilon}(\arccos \mu))= M! \sum_{\nu=1}^{M} \frac{(D^\nu
\phi_{r,\varepsilon})(\arccos \mu)}{\nu!} \sum_{h_1+\dots+h_\nu=M} \frac{%
D^{h_1} \arccos \mu}{h_1!} \dots \frac{D^{h_\nu} \arccos \mu}{h_{\nu!}},
\end{equation}
where the second sum is computed on all the possible integer values of $%
h_1,\dots, h_\nu \geq 1$ with sum equal to M. We note that this sum is
bounded by a constant which depends on $r$; indeed, the arccos is a $C^{\infty}$ function in each compact subset of $(-1,1)$ and since outside $[r-\eps,r]$ all the derivatives of $\phi$ are zero and $r\ne \pi$, $\mu$ is always different from $+1$ and $-1$; hence the second sum of (\ref{faadibruno}) is bounded away from $-1$ and $1$. As far as the first sum is concerned in (\ref{faadibruno}), it is possible to compute it explicitly
\begin{equation}\label{b}
\begin{split}
&\sum_{\nu=1}^{M} \frac{(D^\nu \phi_{r,\varepsilon})(\arccos \mu)}{\nu!}=
\sum_{\nu=1}^{M} \frac{1}{\nu!} D^\nu \bigg(\phi_{r,\varepsilon}\bigg(\frac{%
\arccos \mu-r+\varepsilon}{\varepsilon}\bigg)\bigg)= \\
& =\sum_{\nu=1}^{M} \frac{1}{\nu!} \bigg[D^\nu q_{2k+1}\bigg(\frac{\arccos
\mu-r+\varepsilon}{\varepsilon}\bigg)\bigg]1_{[r-\varepsilon,r]}
=\sum_{\nu=1}^{M} \frac{1}{\nu!} \bigg[D^\nu \sum_{i=0}^{k} B_i^{2k+1} \bigg(%
\frac{\arccos \mu-r+\varepsilon}{\varepsilon}\bigg)\bigg]1_{[r-%
\varepsilon,r]}= \\
& = \sum_{\nu=1}^{M} \frac{1}{\nu!} \bigg[D^\nu \sum_{i=0}^{k} \binom{2k+1}{i%
} \bigg( \frac{\arccos \mu -r +\varepsilon}{\varepsilon}\bigg)^i \bigg(1-%
\frac{\arccos \mu -r+\varepsilon}{\varepsilon}\bigg)^{2k+1-i} \bigg]%
1_{[r-\varepsilon,r]}= \\
& = \sum_{\nu=1}^{M}\frac{1}{\nu!} \frac{1}{\varepsilon^{2k+1}}\bigg[D^\nu
\sum_{i=0}^{k} \binom{2k+1}{i} ( {\arccos \mu -r +\varepsilon})^i (r-\arccos
\mu)^{2k+1-i} \bigg]1_{[r-\varepsilon,r]}.
\end{split}%
\end{equation}
Since $\sum_{i=0}^{k} \binom{2k+1}{i} ( {\arccos \mu -r +\varepsilon})^i
(r-\arccos \mu)^{2k+1-i}$ is a polynomial in the  compact domain $[r-\eps,r]$, we can bound (\ref{b}) by $\dfrac{C_{M,r}}{\eps^{2M+1}}$, where $C_{M,r}$ is a constant depending on $r$ and $M$. The absolute value of (\ref{faadibruno}) satisfy then
\begin{equation}\label{bb}
\bigg|D^M\phi_{r,\varepsilon}(\arccos \mu)\bigg| \leq \dfrac{M!C_{M,r}}{%
\varepsilon^{2M+1}}.
\end{equation}
We are hence in the position to prove Proposition \ref{coeff}.
\begin{proof}[Proof of Proposition \ref{coeff}]
		We recall the following property of the Legendre polynomials (see for instance \cite{abra}, Chapter 22, formula 22.7 and formula 22.8 combined with the Legendre differential equation)
		\begin{equation}\label{nu}
		(2\ell+1)P_\ell(x)= \dfrac{d}{dx}\bigg[P_{\ell+1}(x)-P_{\ell-1}(x)\bigg],
			\end{equation}
		and we substitute it in the definition of $b_{\ell,\eps}^r$	to obtain, integrating by parts,
		\begin{equation}\label{bo}
			\begin{split}
		\int_{-1}^{1} k_{\eps,r}(x) P_\ell(x)\,dx&=\bigg[ k_{\eps,r}(x) \dfrac{P_{\ell+1}(x)-P_{\ell-1}(x)}{2\ell+1} \bigg|_{-1}^{1}- \int_{-1}^{1} \frac{d}{dx}k_{\eps,r}(x) \dfrac{P_{\ell+1}(x)-P_{\ell-1}(x)}{2\ell+1} \,dx\bigg]\\&
		=\dfrac{1}{2\ell+1} \int_{-1}^{1} \frac{d}{dx} k_{\eps,r}(x) P_{\ell+1}(x) \,dx-\dfrac{1}{2\ell+1} \int_{-1}^{1} \frac{d}{dx} k_{\eps,r}(x) P_{\ell-1}(x) \,dx.
			\end{split}
		\end{equation}
		Applying again (\ref{nu}) to $P_{\ell+1}$ and to $P_{\ell-1}$ in the place of $P_\ell$ and integrating by parts, one has that (\ref{bo}) holds
		\begin{equation}
		\begin{split}
		&=\dfrac{1}{2\ell+1} \frac{1}{2\ell+3} \int_{-1}^{1}\frac{d^2}{dx^2}k_{\eps,r}(x) (P_{\ell+2}(x)-P_{\ell}(x))\,dx -\dfrac{1}{(2\ell+1)}\dfrac{1}{2\ell-1}\int_{-1}^{1}\frac{d^2}{dx^2}k_{\eps,r}(x) (P_{\ell}(x)-P_{\ell-2}(x))\,dx\\&
		=\dfrac{1}{2\ell+1} \frac{1}{2\ell+3} \int_{-1}^{1}\frac{d^2}{dx^2}k_{\eps,r}(x) P_{\ell+2}(x)\,dx +\dfrac{1}{(2\ell+1)}\dfrac{1}{2\ell-1}\int_{-1}^{1}\frac{d^2}{dx^2}k_{\eps,r}(x) P_{\ell-2}(x)\,dx\\&
		 -\dfrac{1}{2\ell+1} \frac{1}{2\ell+3} \int_{-1}^{1}\frac{d^2}{dx^2}k_{\eps,r}(x)P_{\ell}(x)\,dx -\dfrac{1}{2\ell+1}\dfrac{1}{2\ell-1}\int_{-1}^{1}\frac{d^2}{dx^2}k_{\eps,r}(x)P_{\ell}(x)\,dx.
		\end{split}
		\end{equation}
Iterating $M$ times, taking the absolute value, using (\ref{bb}) and the fact that $|P_\ell(x)|\leq1$ in $[-1,1]$ $\forall \ell$, one has that $|\int_{-1}^{1} k_{\eps,r}(x) P_\ell(x)\,dx|$ is bounded by $2^{M}$ terms times 
		$$ \dfrac{ C}{\ell^M}\dfrac{M!C_{M,r}}{\eps^{2M+1}}.$$
	Consequently, for $\ell \geq 1$
\begin{equation}\label{perche cond1}
|b_{\ell;\eps}^r|\leq \sqrt{\dfrac{2\ell+1}{4\pi}} \dfrac{C 2^M M!C_{M,r}}{\eps^{2M+1} \ell^{M}} \leq \dfrac{K_{M,r}}{\ell^{M-1/2} \eps^{2M+1}},
\end{equation}
where $K_{M,r}= \sqrt{\dfrac{3}{4\pi}} M!2^{M} C C_{M,r}$.


\end{proof}
\begin{oss}
Note that for the coefficients $b_{\ell;\eps}^r$ to go to zero, the condition $$\ell^{M-1/2}\eps^{2M+1} \rightarrow \infty,$$ as $\ell \rightarrow \infty$ and $\eps \rightarrow 0$, has to be satisfied; Assumption $\ref{cond-eps}$ ensures it. 
\end{oss}
In conclusion, this section can be summarized in the theorem below.

\begin{thm}
\label{mollifier} Let $B \subset \mathbb{S}^2$ be a spherical cap of radius $r \in (0,\pi)$, parametrized
by $\theta \in [0,r], \varphi \in [0,2\pi]$. For
any $M>0 \in \mathbb{N} $ and $\varepsilon>0$ there exists a function $%
1_{B,\varepsilon} \in C^{M}$ which converges to the indicator function $%
1_B(x)$ in $L^1(\mathbb{S}^2)$, as $\varepsilon \rightarrow 0$, such that the
coefficients $b_{\ell,\varepsilon}^r$ of the Fourier expansion 
	\begin{equation}\label{numeretto4}
1_{B,\varepsilon}( \theta)=\sum_{\ell=0}^{\infty} b^r_{\ell;\varepsilon} 
\sqrt{\dfrac{2\ell+1}{4\pi}} P_\ell(\cos \theta), \mbox{ } \mbox{ }%
b^r_{\ell;\varepsilon}=\sqrt{\frac{2\ell+1}{4\pi}}\int_{-1}^{1}
1_{B,\varepsilon}(\arccos x)Y_\ell(x) \, dx
\end{equation}
satisfy the condition 
	\begin{equation}\label{numeretto3}
|b^r_{\ell;\varepsilon}| \leq \min \bigg\{ b_{0,\eps}^r, \dfrac{K_{M,r}}{ \ell^{M-\frac{1}{2}%
	}\varepsilon^{2M+1}} \bigg\}
\end{equation}
as $\ell \rightarrow \infty$, where 
\begin{equation*}
K_{M,r}= \sqrt{\dfrac{3}{4\pi}} M!2^{M} C C_{M,r}.
\end{equation*}
\end{thm}

\begin{ex}
Let us consider $k=1$, then $n=2k+1=3$, $M=1$ and $B_0(t)=(1-t)^3,$ $B_1(t)= 3t(1-t)^2$. It
follows that 
\begin{equation*}
q(t)=B_0(t)+B_1(t)= 2t^3-3t^2+1
\end{equation*}
and 
\begin{equation*}
q^{\prime}(t)=6t^2-6t.
\end{equation*}
Hence, the first derivative of $k_{\eps,r}(\mu), \mu \in [-1,1]$ is
\begin{equation}
\begin{split}
\dfrac{d}{d\mu} k(\mu)&=\dfrac{d}{d\mu} \phi(\arccos\mu)= \phi^{\prime
}(\arccos\mu) \dfrac{-1}{\sqrt{1-\mu^2}}\\&=\bigg[\frac{6}{\varepsilon^3}
(\arccos \mu-r+\varepsilon)^2-\frac{6}{\varepsilon^3} (\arccos
\mu-r+\varepsilon)\bigg] \dfrac{-1}{\sqrt{1-\mu^2}} \\
& =\frac{6}{\varepsilon^3} (\arccos \mu-r+\varepsilon)(\arccos \mu-r) \dfrac{%
-1}{\sqrt{1-\mu^2}}.
\end{split}%
\end{equation}
Accordingly for $b_{\ell,\eps}^r$ one obtains
\begin{equation*}
|b_{\ell,\varepsilon}|\leq \bigg|\dfrac{1}{2\ell+1} \sqrt{\dfrac{2\ell+1}{4\pi}}
\int_{-1}^{1} \dfrac{d}{d\mu} k(\mu) (P_{\ell+1}(x)-P_{\ell-1}(x))\,dx\bigg|\leq 
\dfrac{ C_{r}}{\ell^{1/2}\varepsilon^3}.
\end{equation*}
\end{ex}
We give some values of $b_{\ell,\varepsilon}^r$ in figure \ref{fig:fig2.jpg};
the graphic was realized choosing the parameters as $\varepsilon=\frac{1}{2}$ and $r=\frac{\pi}{4}
$. 

\begin{figure}[h!]
\label{fig2.jpg} \centering
\includegraphics[width=0.8\linewidth]{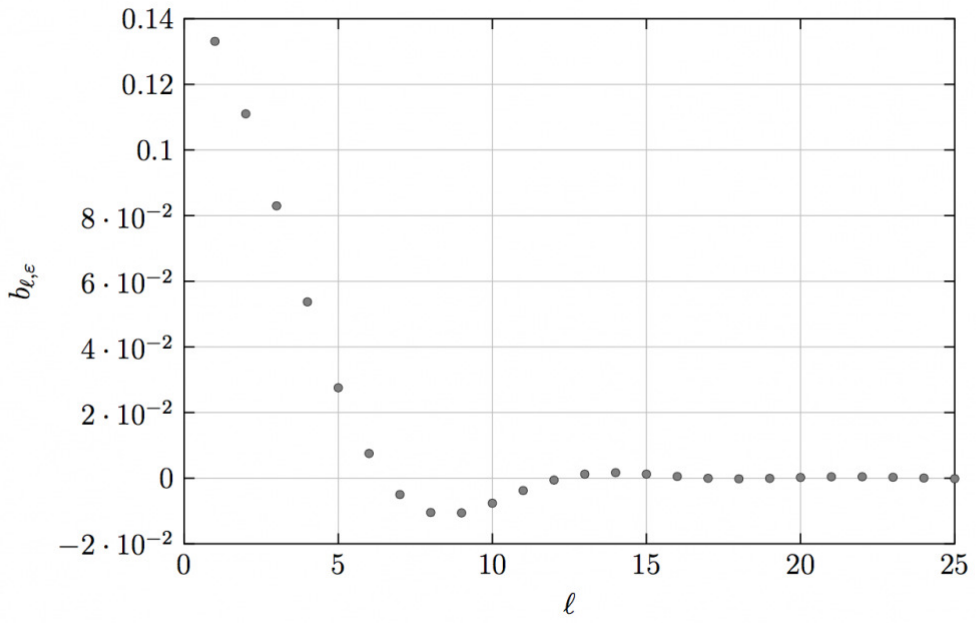}
\caption{First values of $b_{\ell,1/2}^{\protect\pi/4}$ varying $\ell$.}
\label{fig:fig2.jpg}
\end{figure}

\begin{ex}\label{ex2}
Choosing $k=2$, one has $n=5$, $M=2$ and $B_0(t)=(1-t)^5$, $B_1(t)=5t(1-t)^4$
and $B_{2}(t)=10t^2(1-t)^3$. One finds that 
\begin{equation*}
q(t)=-6t^5+15t^4-10t^3+1,
\end{equation*}
\begin{equation*}
q^{\prime}(t)=-30 t^4+60t^3-30t^2
\end{equation*}
and 
\begin{equation*}
q^{\prime \prime} (t)= -120 t^3+180t^2-60t.
\end{equation*}
Same computations to the previous example give

\begin{equation*}
\bigg|\dfrac{d^2}{d\mu^2} k(\mu)\bigg| \leq \frac{C_r}{\varepsilon^5}
\end{equation*}
and 
\begin{equation*}
|b_{\ell,\varepsilon}| \leq \dfrac{C_{r}}{\ell^{3/2}\varepsilon^5}.
\end{equation*}
\end{ex}

\begin{oss}
In the table and the graphs below, we compare $b_{\ell,\eps }^r$, for $\ell =1,2,3,4,5$, for different values of $%
\varepsilon $ and the assumptions of the example \ref{ex2}
\begin{center}
\begin{tabular}{|c|c|c|c|c|}
\hline
$\ell $ & $b_{\ell ;1/2}$ & $b_{\ell ,1/4}$ & $b_{\ell ,1/8}$ & $b_{\ell
,1/10}$ \\ \hline
1 & 0.132269 & 0.188425 & 0.218866 & 0.225059 \\ \hline
2 & 0.111278 & 0.147981 & 0.163897 & 0.166747 \\ \hline
3 & 0.0843363 & 0.0983641 & 0.0987674 & 0.0982093 \\ \hline
4 & 0.0557163 & 0.0493925 & 0.0381274 & 0.0352638 \\ \hline
5 & 0.0294925 & 0.00959262 & -0.00638063 & -0.0097985 \\ \hline
\end{tabular}
\end{center}
\begin{figure}[h]
\centering
\includegraphics[width=0.7\linewidth]{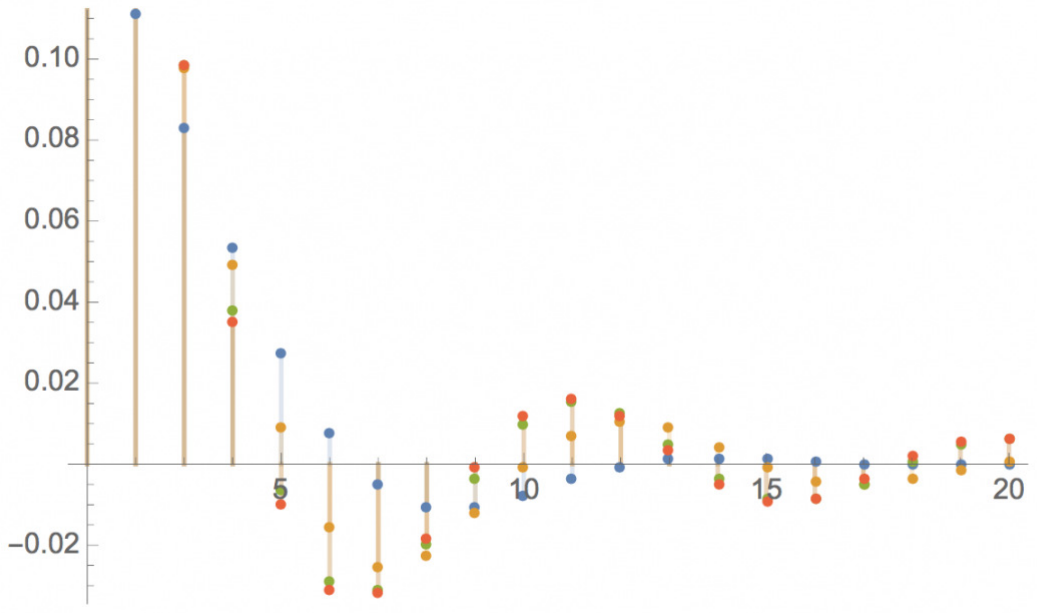}
\caption{$b_{\ell,\eps}^r$ varying $\ell$, for $\eps=\frac{1}{2},\frac{1}{4},\frac{1}{8},\frac{1}{10}$.}
\label{fig:grafbl2.jpg}
\end{figure}
\begin{figure}[h!]
\centering
\includegraphics[width=0.7\linewidth]{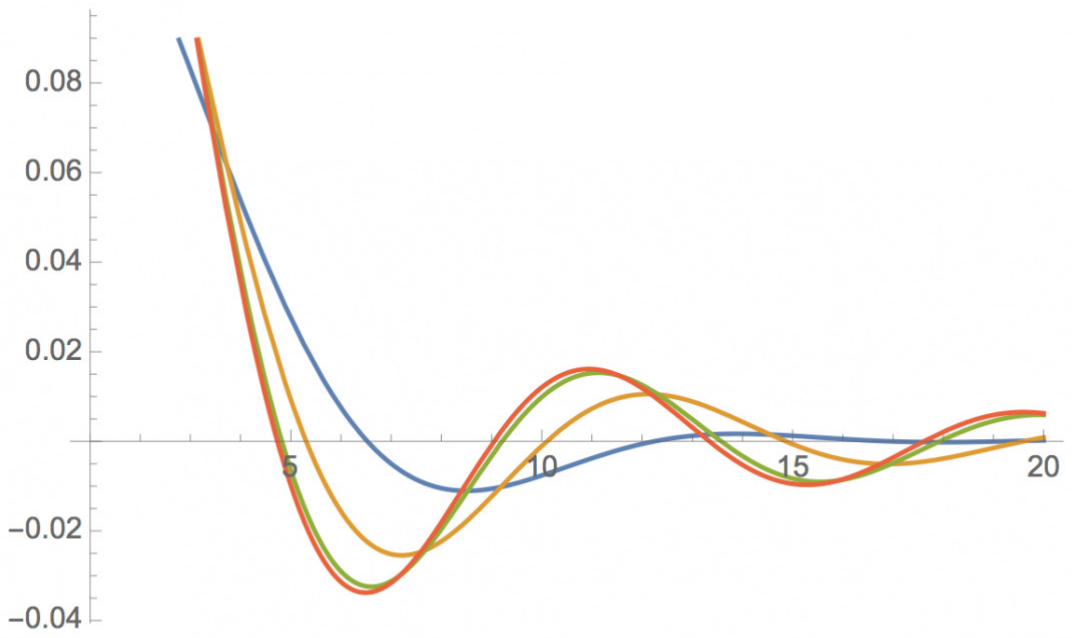}
\caption{{}}
\label{fig:grafbl.jpg}
\end{figure}
We note that the decay of the coefficients $b_{\ell,\eps}$ is actually faster than the one given by our upper bound. 

\end{oss}

\begin{oss}	It is quite natural to compare our result with the work of Lang and Schwab in \cite{Lang e Schwab}. We report here, briefly, their findings. Hence, they define the space $V^n(-1,1)$ as the closures of $H^n(-1,1)$, where $H^n(-1,1)$ is the standard Sobolev spaces, with respect to the weighted norms $||u||_{V^{n}(-1,1)}^2:=
	\sum_{j=0}^{n} |u|^2_{V^j(-1,1)},$ where for $j \in \mathbb{N}_0$,
	\begin{equation*}
	|u|^2_{V^j(-1,1)}:= \int_{-1}^{1} \bigg|\dfrac{\partial^j}{\partial \mu^j}
	u(\mu)\bigg|^2 (1-\mu^2)^j \,d \mu,
	\end{equation*}
	is a seminorn. Denoted as $(\frac{2\ell+1}{2}(1+\ell^{2n}), \ell \in \mathbb{N}_0)$ the sequence of weights, the authors in \cite{Lang e Schwab} show an
	isomorphism between the spaces $V^n(-1,1)$ and the
	spaces of the weights $\ell_n:= \ell^2((\frac{2\ell+1}{2}(1+\ell^{2n}), \ell \in \mathbb{N})$. 
	Precisely, they proved that, 
	for $u(\mu) \in V^n(-1, 1), n \in 
	\mathbb{N}_0, $ the sequence $%
	(\ell^{n+1/2}A_\ell, \ell \geq n)$, with $A_\ell=2\pi u_\ell$, is in $\ell^2(%
	\mathbb{N}_0)$ if and only if $(1-\mu^2)^{n/2} \frac{\partial^n}{\partial \mu^n }u(\mu)$ is
	in $L^2(-1, 1);$ namely,
	\begin{equation*}
	\dfrac{1}{(4\pi)^2} \sum_{\ell \geq n} A_\ell^2\frac{2\ell+1}{2}
	\ell^{2n}<+\infty
	\end{equation*}
	if and only if 
	\begin{equation*}
	\int_{-1}^{1} \bigg|\dfrac{\partial^n}{\partial\mu^n}u(\mu) \bigg|^2(1-\mu^2)^n \,d\mu
	<\infty.
	\end{equation*}
	More explicitly, in their proof (p. 13 \cite{Lang e Schwab}) they get that 
	\begin{equation}\label{LeS}
	\int_{-1}^{1} \bigg|\dfrac{\partial^n}{\partial\mu^n}u(\mu) \bigg|^2(1-\mu^2)^n \,d\mu=
	\sum_{\ell \geq n} A_\ell^2 \dfrac{2\ell+1}{2(4\pi)^2} \frac{(\ell+n)!}{%
		(\ell-n)!}
	\end{equation}
	and 
	\begin{equation*}
	c_1(n)\ell^{2n} \leq \frac{(\ell+n)!}{(\ell-n)!} \leq c_2(n) \ell^{2n}.
	\end{equation*}
	Although it is possible to compute explicitly the integral on the left hand side of (\ref{LeS}), this would be sufficient only for a bound on the tail behavior of the series in the right hand side of (\ref{LeS}), while, in our situation, we require a full control on any term $A_\ell^2$.
\end{oss}

\begin{oss} We refer to \cite{LMX} for the broadly similar construction of a ``spherical bump function$"$. Also, our proposal is in some sense symmetric to so-called needlets (see i.e., \cite{NPW}, \cite{NPW2}, \cite{BKMP2009} and Chapter 10 of \cite{M e Peccati}). Indeed, in the standard needlet construction one considers spherical functions with compact support in the harmonic domain and nearly-exponential decay in the real domain, whereas here the converse is studied: functions with compact support in the real domain and polynomial decays in the harmonic space.
\end{oss}

\section{Proof of the main result}
Here we finally prove Theorem \ref{tlc}; as stated at the beginning of the paragraph, we do that studying each single term of the chaotic projection in (\ref{serie-var}) separately. We divide in small different subsections the results obtained for these components.\\\\
From now on, $1_{B,\eps}(x)$ is the function given in Remark \ref{approx}, satisfying (\ref{numeretto3}) and Assumption \ref{cond-eps}.

\subsubsection{ First chaotic component}

The variance of the first chaotic component, i.e., Proposition \ref{cor1chaos} follows as a corollary of the lemma below.

\begin{lem}
	\label{1chaos} For any $\varepsilon>0,$ satisfying Assumption \ref{cond-eps},
	\begin{equation}\label{varianza1}
	\Var\bigg(\int_{\mathbb{S}^2} 1_{B}(x)T_\ell(x) \, dx\bigg)=\dfrac{4\pi}{2\ell+1}
	b_{\ell;\varepsilon}^2+O(\ell^{-1/2}\varepsilon^{3/2}),
	\end{equation}
	as $\ell \rightarrow 0,$ where $b_{\ell;\varepsilon}$ are the Fourier
	coefficients of $1_{B,\eps}(x)$, given by (\ref{numeretto4}).
\end{lem}

\begin{proof}[Proof of Lemma \ref{1chaos}]
The first chaotic projection can be written as 

$$\int_B T_\ell (x) \, dx= \int_{\mathbb{S}^2} [1_B(x)-1_{B;\eps}(x)]T_\ell(x) \,dx+\int_{\mathbb{S}^2} 1_{B;\eps}(x)T_\ell(x) \,dx$$
and consequently, its variance as
\begin{equation}\label{a1}
	\begin{split}
	\Var \bigg(\int_B T_\ell (x) \, dx\bigg)&=\Var \bigg( \int_{\mathbb{S}^2} [1_B(x)-1_{B;\eps}(x)]T_\ell(x) \,dx \bigg) +\Var \bigg(\int_{\mathbb{S}^2} 1_{B;\eps}(x)T_\ell(x) \,dx\bigg)\\&
	+2\mathbb{E}\bigg[ \int_{\mathbb{S}^2 \times \mathbb{S}^2 } (1_B(x)-1_{B;\eps}(x))1_{B;\eps} (y) T_\ell(x)T_\ell(y) \, dx \,dy \bigg].
	\end{split}
\end{equation}
For the first variance of (\ref{a1}) it holds that
\begin{equation}\label{two}
	\begin{split}
\Var\bigg(\int_{\mathbb{S}^2} (1_B(x)-1_{B;\eps}(x))T_\ell(x) \,dx\bigg)& = \int_{\mathbb{S}^2 \times \mathbb{S}^2 } (1_B(x)-1_{B;\eps}(x)) (1_B(y)-1_{B;\eps}(y)) \mathbb{E}[T_\ell(x)T_\ell(y)]\, dx dy\\&
\leq \int_{\mathbb{S}^2}|1_B(x)-1_{B;\eps}(x) | \bigg( \int_{\mathbb{S}^2}| 1_B(y)-1_{B;\eps}(y) | |P_\ell(\langle x,y \rangle)|\,dy \bigg)\,dx 
	\end{split}
\end{equation}
and applying the Cauchy-Schwarz inequality to the second integral, (\ref{two}) is bounded by 
\begin{equation}\label{non}
\begin{split}
&\leq \int_{\mathbb{S}^2}|1_B(x)-1_{B;\eps}(x) | \bigg( \int_{\mathbb{S}^2}| 1_B(y)-1_{B;\eps}(y) |^2\,dy \bigg)^{1/2} \bigg( \int_{\mathbb{S}^2}| P_\ell(\langle x,y \rangle)|^2\,dy \bigg)^{1/2}\,dx \\& \leq \sqrt{\dfrac{2}{2\ell+1}} 2\pi \sqrt{2\pi} \eps \sqrt{\eps};
\end{split} 
\end{equation}
the third term in (\ref{a1}) is as small as this one by Cauchy-Schwarz inequality.
Concerning the second variance in (\ref{a1}), one has
\begin{equation}\label{a2}
\begin{split}
\mbox{Var}\bigg( \int_{\mathbb{S}^2} 1_{B;\eps}(x) T_\ell (x) \, dx \bigg)&= \mathbb{E}\bigg[ \bigg(\int_{\mathbb{S}^2}  1_{B;\eps}(x) T_\ell (x) \, dx\bigg)^2\bigg]= \int_{\mathbb{S}^2\times \mathbb{S}^2}  1_{B;\eps}(x)  1_{B;\eps}(y) \mathbb{E}[T_\ell (x) T_\ell(y)] \,dx\,dy\\&
=\int_{\mathbb{S}^2 \times \mathbb{S}^2} 1_{B;\eps}(x) 1_{B;\eps}(y)P_\ell (\langle x,y\rangle) \,dx \,dy.
\end{split}
\end{equation}
Through the Addition formula (\cite{M e Peccati} p. 66):
	$$\sum_{m=-\ell}^{\ell} \bar{Y}_{\ell m}(x) Y_{\ell m}(y)=\dfrac{2\ell+1}{4\pi} P_\ell(\langle x,y \rangle)$$
	 and the expansion
	 \begin{equation}\label{alessio3}
	 	1_{B,\varepsilon}(x)=\sum_{\ell=0}^{\infty} b_{\ell;\varepsilon} Y_{\ell 0}(x),
	 		 \end{equation}
	 		  it is possible to write (\ref{a2}) as
	\begin{equation}\label{ss}
	\int_{\mathbb{S}^2 \times \mathbb{S}^2} \dfrac{4\pi}{2\ell+1} \sum_{m=-\ell}^{\ell} \overline{Y_{\ell m}(x)} Y_{\ell m}(y) \sum_{\ell_1=0}^{\infty} b_{\ell_1;\eps} Y_{\ell_1 0}(x) \sum_{\ell_2=0}^{\infty} b_{\ell_2;\eps} Y_{\ell_2 0}(y) \,dx \,dy.
	\end{equation}
Condition (\ref{cond2}) implies that the series $\sum_{\ell=0} b_{\ell,\eps}Y_{\ell 0}(x)$ is absolutely convergent; indeed
	$$\sum |b_{\ell,\eps}^r||Y_{\ell}(x)|\sim \sum |b_{\ell,\eps}^r| \sqrt{\ell}< \sum \frac{1}{\ell^2}<\infty,$$
	 and so we can exchange the series with the integral to derive that (\ref{ss}) equals to 
	\begin{equation}\label{equazione}
	\begin{split}
	& \dfrac{4\pi}{2\ell+1}  \sum_{m=-\ell}^{\ell} \sum_{\ell_1=0}^{\infty} b_{\ell_1;\eps} \sum_{\ell_2=0}^{\infty} b_{\ell_2;\eps} \int_{\mathbb{S}^2 \times \mathbb{S}^2} \overline{Y_{\ell m}(x)} Y_{\ell m}(y) Y_{\ell_1 0}(x)  Y_{\ell_2 0}(y) \,dx \,dy=\\&
	=\dfrac{4\pi}{2\ell+1} \sum_{m=-\ell}^{\ell} \sum_{\ell_1=0}^{\infty} b_{\ell_1;\eps} \sum_{\ell_2=0}^{\infty} b_{\ell_2;\eps} \int_{\mathbb{S}^2} \overline{Y_{\ell m}(x)} Y_{\ell_1 0}(x) \,dx \int_{\mathbb{S}^2} Y_{\ell m}(y) Y_{\ell_2 0}(y)\,dy.
	\end{split}
	\end{equation}
	The orthogonality condition (\cite{M e Peccati} eq. (3.39) p. 66)
	\begin{equation}
	\int_{\mathbb{S}^2}  \overline{Y_{\ell m}(x)} Y_{\ell' m'}(x) \,dx = \delta_{\ell'}^{\ell} \delta_{m'}^{m},
	\end{equation}
reduces (\ref{equazione}) to
\begin{equation}\label{ics}
\dfrac{4\pi}{2\ell+1}\sum_{m=-\ell}^{\ell} \sum_{\ell_1=0}^{\infty} b_{\ell_1,\eps} \sum_{\ell_2=0}^{\infty} b_{\ell_2;\eps} \delta_{\ell_1}^{\ell} \delta_{0}^{m}\delta_{\ell_2}^{\ell} \delta_{0}^{m}= \dfrac{4\pi}{2\ell+1} b_{\ell;\eps}^2 
\end{equation}
and then the variance (\ref{a2}) is
	$$\Var\bigg(\int_{\mathbb{S}^2} 1_{B;\eps}(x) T_\ell (x) \, dx \bigg)=\dfrac{4\pi}{2\ell+1} b_{\ell;\eps}^2;$$
thus (\ref{ics}), (\ref{two}) lead to the thesis of the lemma.
\end{proof}
Now, Proposition \ref{cor1chaos} is proven by choosing a sequence $\eps=\eps_{\ell}$ satisfying Assumption \ref{cond-eps}.


\subsubsection{Second chaotic component}
In this subsection we prove Proposition \ref{2chaotic}; to this aim we introduce the two lemmas below, whose proofs can be found in Appendix A

\begin{lem}\label{lemma1} Under the assumptions of Proposition \ref{2chaotic}, one has that
\begin{equation}\label{smoothvar}
\Var\bigg(\int_{\mathbb{S}^2} 1_{B;\varepsilon}(x) H_2(T_\ell(x)) \,dx \bigg)=8
\pi \sum_{\ell_1} b_{\ell_1;\varepsilon}^2 \dfrac{1}{2\ell_1+1} \bigg(%
C_{\ell 0 \ell 0}^{\ell_1 0}\bigg)^2,
\end{equation}
where $\{C_{\ell 0 \ell 0}^{\ell_1 0}\}$ are the Clebsch-Gordan coefficients 
(see \cite{quantum theory} or the Appendix). 
\end{lem}

\begin{lem}\label{lemma2}
There exist two strictly positive constants $c_1$ and $c_2$ such that 
\begin{equation}
\begin{split}
\dfrac{c_1}{\ell}\leq \Var\bigg(\int_{\mathbb{S}^2}1_{B;\varepsilon}(x)
H_2(T_\ell(x)) \,dx \bigg)&=8 \pi \sum_{\ell_1} b_{\ell_1;\varepsilon}^2 
\dfrac{1}{(2\ell+1)} \bigg(C_{\ell 0 \ell 0}^{\ell_1 0}\bigg)^2 \leq \dfrac{%
c_2}{\ell}.
\end{split}%
\end{equation}
as $\ell \rightarrow \infty.$
\end{lem}

\begin{proof}[Proof of Proposition \ref{2chaotic}]

The variance of the second chaotic component can be written as


	\begin{equation}\label{one}
		\begin{split}
	\Var\bigg[ \int_{\mathbb{S}^2} 1_B(x) H_2(T_\ell(x)) \,dx \bigg]&=\mathbb{E}\bigg[\int_{\mathbb{S}^2} (1_B(x)-1_{B;\varepsilon}(x) )H_2(T_\ell(x)) \,dx\bigg]^2+\Var\bigg[\int_{\mathbb{S}^2} 1_{B,\varepsilon}(x) H_2(T_\ell(x)) \,dx\bigg]\\&
	+2\mathbb{E}\bigg[\int_{\mathbb{S}^2 \times \mathbb{S}^2 } 1_{B,\varepsilon}(x) \bigg(1_B(y)-1_{B,\varepsilon}(y)\bigg) H_2(T_\ell(x))H_2(T_\ell(y)) \,dxdy\bigg].
		\end{split}
	\end{equation}
The first integral in (\ref{one}), in view of Remark 4.10 in \cite{M e Peccati}, is
	\begin{equation}\label{key2}
		\begin{split}
	&\mathbb{E}\bigg[\int_{\mathbb{S}^2} (1_B(x)-1_{B;\varepsilon}(x) )H_2(T_\ell(x)) \,dx\bigg]^2=\\&=\int_{\mathbb{S}^2 \times \mathbb{S}^2} (1_B(x)-1_{B;\eps}(x))(1_B(y)-1_{B;\eps}(y)) \mathbb{E}[H_2(T_\ell(x))H_2(T_\ell(y))] \,dx dy\\&
	\leq 2 \int_{\mathbb{S}^2 \times \mathbb{S}^2} |1_B(x)-1_{B;\eps}(x)|   P^2_\ell(\langle x,y \rangle ) \, dxdy \leq 2 \int_{\mathbb{S}^2 } |1_{B;\eps}(x)-1_B(x)| \cdot \int_{\mathbb{S}^2} P^2_\ell(\langle x,y\rangle) \,dy \,dx \\&
	\leq 2 C \eps \dfrac{2}{2\ell+1}, 
		\end{split}
	\end{equation}
	where $C=2\pi$ has already been computed in (\ref{two}), and for the Cauchy-Schwarz inequality, the same bound holds for the third integral in (\ref{one}). Then, Lemma \ref{lemma1} with Lemma \ref{lemma2} conclude the proof.
\end{proof}


\subsubsection{ Terms of the chaotic components for $q\geq 3$}\label{chapter q}
Let us consider the chaotic components of order $q$ for $q\geq3$.\\
The variance of the third term in (\ref{serie-var}) can be bounded by its absolute value, and hence we can bound the integral on $B$ with the one computed on $\mathbb{S}^2$ as in the following way, we have that
\begin{equation}\label{sss}
\begin{split}
\mbox{Var} \bigg( \int_B \dfrac{J_3(z)}{3!}&H_3(T_\ell(x)) dx\bigg)= \dfrac{%
	J_3(z)^2}{3!} \int_B \int_B P_\ell(\langle x,y \rangle)^3 \,dx dy \leq 
\dfrac{J_3(z)^2}{3!} \int_{B} \int_{\mathbb{S}^2} | P_\ell(\langle x,y \rangle)|^3
\,dx dy \\
& =\dfrac{J_3(z)^2}{3!} 2 \pi m(B) \int_{0}^{\pi/2} | P_\ell(\cos \theta)|^3
\sin \theta \,d\theta =\dfrac{J_3(z)^2}{3!} 2 \pi m(B) \int_{0}^{1} |
P_\ell(x)|^3 \,dx;
\end{split}%
\end{equation}
the Cauchy-Schwartz inequality implies that (\ref{sss}) is 
\begin{equation}\label{alessio1}
\leq \dfrac{J_3(z)^2}{3!} 2\pi m(B) \bigg( \int_{0}^{1}P_\ell( x)^2 \,dx %
\bigg)^{1/2} \bigg(\int_{0}^{1} P_\ell( x)^4 \,dx \bigg)^{1/2}
\end{equation}
and since it has been proved in \cite{M e W 2011} and \cite{M e Mau 2015} that $\int_{0}^{1}P_\ell( x)^2 \,dx =O\big(\frac{1}{\ell}\big)$ and $\int_{0}^{1} P_\ell( x)^4 \,dx= O\big( \frac{\log \ell}{\ell^2}\big)$,
(\ref{alessio1}) has order $O\big(\frac{\sqrt{\log \ell}}{\ell \sqrt{\ell}}\big),$ as $\ell \rightarrow
\infty.$\\\\
Likewise, for the variance of the fourth chaotic projection in (\ref{serie-var}), 
we obtain that

\begin{equation*}
\begin{split}
\Var \bigg( \int_B \dfrac{J_4(z)}{4!} H_4(T_\ell(x)) dx\bigg)&= \dfrac{%
	J_4(z)^2}{(4!)^2} \int_B \int_B P_\ell(\langle x,y \rangle)^4 dx dy \\
& \leq \dfrac{J_4(z)^2}{(4!)^2} \int_{B} \int_{\mathbb{S}^2} P_\ell(\langle x,y
\rangle)^4 dx dy \\
& = \dfrac{J_4(z)^2}{(4!)^2} m(B) 2\pi \int_{0}^{1} P_\ell(x)^4 \,dx
\end{split}%
\end{equation*}
which behaves as $\dfrac{\log \ell}{\ell^2},$
as $\ell \rightarrow \infty$ \cite{M e W 2011}.\\\\
Eventually, for the remaining terms in (\ref{serie-var}), in the same way we get
\begin{equation*}
\begin{split}
\Var \bigg(&\int_B \sum_{q=5}^{\infty} \dfrac{J_q(z)}{q!}
H_q(T_\ell(x)) \,dx \bigg)=\mathbb{E}\bigg[ \int_B \sum_{q=5}^{\infty} \dfrac{J_q(z)}{%
	q!}H_q(T_\ell(x)) \,dx \bigg]^2 \\
&= \sum_{q=5}^{\infty} \dfrac{J_q(z)^2}{(q!)^2}\int_{B\times B}
\mathbb{E} [H_q(T_\ell(x)) H_q(T_\ell(y))] \,dx dy = \sum_{q=5}^{\infty} \dfrac{J_q(z)^2%
}{(q!)^2} \int_{B\times B} q! P_\ell( \langle x,y\rangle)^q \,dx dy \\
& \leq \sum_{q=5}^{\infty} \dfrac{J_q(z)^2}{q!} \int_{B\times B} |P_\ell(
\langle x,y\rangle)|^q \,dx dy \leq \sum_{q=5}^{\infty} \dfrac{J_q(z)^2}{q!}
\int_{B\times \mathbb{S}^2} |P_\ell( \langle x,y\rangle)|^q \,dx dy \\
& \leq \sum_{q=5}^{\infty} \dfrac{J_q(z)^2}{q!} 2\pi m(B) \int_{0}^{\pi/2}
|P_\ell(\cos \theta)|^q \sin \theta \, d\theta =\sum_{q=5}^{\infty} \dfrac{%
	J_q(z)^2}{q!} 2\pi m(B) \int_{0}^{1} |P_\ell(x)|^q\, dx
\end{split}%
\end{equation*}
and $ \int_{0}^{1} |P_\ell(x)|^q\, dx =O\big(\frac{1}{\ell^2}\big)$ (\cite{M e W 2011bis}, Lemma 5.7 or \cite{M e Mau 2015}, Proposition 1.1).

\subsubsection{Quantitative Central Limit Theorem}
In this subsection, we finally prove Theorem \ref{tlc}, assuming Proposition \ref{4 cumulant}; the argument is quite similar to the one for the full sphere given in \cite{M e Mau 2015}. 
\begin{proof}[Proof of Theorem \ref{tlc} assuming Proposition \ref{4 cumulant}]
	As in \cite{M e Mau 2015}, we denote 
	\begin{equation*}
	S_\ell(M):=\int_B M(T_\ell(x)) \,dx,
	\end{equation*}
	with 
	\begin{equation*}
	M(T_\ell(x)):=1_{((T_\ell(x))>z)}(T_\ell(x)).
	\end{equation*}
	Now, we consider the chaotic expansion
	$$S^\prime_\ell(M):=S_\ell(M)-\mathbb{E}[S_\ell(M)]=\int_{B} \sum_{q=1}^{\infty} \dfrac{J_q(M)H_q(T_\ell(x))}{q!} \,dx,$$
	which we write as
	\begin{equation*}
	\begin{split}
	S^\prime_\ell(M)&=J_1(M)h_{\ell;1}(B)+\dfrac{J_2(M)}{2}h_{\ell;2}(B)+\dfrac{J_3(M)}{3!}h_{\ell;3}(B)+\dfrac{J_4(M)}{4!}h_{\ell;4}(B)+\int_{B} \sum_{q=5}^{\infty} \dfrac{J_q(M)H_q(T_\ell(x))}{q!} \, dx\\
	&=S_\ell(M,1)+S_\ell(M,2)
	\end{split}
	\end{equation*}
	where
	$$S_\ell(M;1)=J_1(M)h_{\ell;1}(B)+\dfrac{J_2(M)}{2}h_{\ell;2}(B)+\dfrac{J_3(M)}{3!}h_{\ell;3}(B)+\dfrac{J_4(M)}{4!}h_{\ell;4}(B),$$
	$$S_\ell(M;2)=\int_B \sum_{q=5}^{\infty} \dfrac{J_q(M)H_q(T_\ell(x))}{q!} \, dx.$$
	Hence, one has that
	\begin{equation}
	\begin{split}
	&d_W\bigg(\dfrac{S^\prime_\ell(M)}{\sqrt{\Var[S_\ell(M)]}}, \mathcal{N}(0,1)\bigg)
	\leq d_W\bigg(\dfrac{S^\prime_\ell(M)}{\sqrt{\Var[S_\ell(M)]}}, \dfrac{S_\ell(M;1)}{\sqrt{\Var[S_\ell(M)]}}\bigg)+\\&+d_W\bigg(\dfrac{S_\ell(M;1)}{\sqrt{\Var[S_\ell(M)]}}, \mathcal{N} \bigg(0, \dfrac{\Var[S_\ell(M;1)]}{\Var[S_\ell(M)]}\bigg)\bigg)+
	d_W\bigg( \mathcal{N} \bigg(0, \dfrac{\Var[S_\ell(M;1)]}{\Var[S_\ell(M)]}\bigg), \mathcal{N}(0,1) \bigg) \\&
	\leq \dfrac{1}{\sqrt{\Var[S_\ell(M)]}} \mathbb{E}\bigg[\bigg(\int_{B} \sum_{q=5}^{\infty} \dfrac{J_q(M)H_q(T_\ell(x))}{q!} \,dx\bigg)^2\bigg]^{1/2}+\\&
	+ d_W\bigg(\dfrac{S_\ell(M;1)}{\sqrt{\Var[S_\ell(M)]}}, \mathcal{N}\bigg(0, \dfrac{\Var[S_\ell(M;1)]}{\Var[S_\ell(M)]}\bigg)\bigg)+d_W\bigg( \mathcal{N}\bigg(0, \dfrac{\Var[S_\ell(M;1)]}{\Var[S_\ell(M)]}\bigg), \mathcal{N}(0,1) \bigg).
	\end{split}
	\end{equation}
	We have seen that
	$$\Var(S_\ell(M;2)) \ll \frac{1}{\ell^2}$$ and since $\Var(S_\ell(M))$ has the same asymptotic order as the second chaotic component, we have that
	$$\frac{\Var(S_\ell(M;2))}{\Var(S_\ell(M))}\ll \dfrac{1}{\ell};$$
	moreover, the triangular inequality gives 
	\begin{equation}\label{alessio2}
	\begin{split}
	d_W\bigg(\dfrac{S_\ell(M;1)}{\sqrt{\Var(S_\ell(M))}}, &\mathcal{N} \bigg(0, \dfrac{\Var(S_\ell(M;1))}{\Var(S_\ell(M))}\bigg)\bigg) \leq d_W\bigg( \frac{J_2(M)}{2\sqrt{\Var(S_\ell(M))}}h_{\ell;2}(B), \mathcal{N}\bigg(0, \dfrac{\Var(S_\ell(M;1))}{\Var(S_\ell(M))}\bigg) \bigg)+\\&
	+d_W\bigg( \frac{J_1(M)}{\sqrt{\Var(S_\ell(M))}}h_{\ell;1}(B)+\frac{J_2(M)}{2\sqrt{\Var(S_\ell(M))}}h_{\ell;2}(B)+\\&+\frac{J_3(M)}{3!\sqrt{\Var(S_\ell(M))}}h_{\ell;3}(B)+\frac{J_4(M)}{4!\sqrt{\Var(S_\ell(M))}}h_{\ell;4}(B), \mathcal{N} \bigg(0, \dfrac{\Var(S_\ell(M;1))}{\Var(S_\ell(M))}\bigg) \bigg).
	\end{split}
	\end{equation}
	
	For the first term in (\ref{alessio2}), we can use, again, the triangular inequality to obtain that
	\begin{equation}\label{fantasia}
	\begin{split}
	&d_W\bigg( \frac{J_2(M)}{2\sqrt{\Var(S_\ell(M))}}h_{\ell;2}(B), \mathcal{N}\bigg(0, \dfrac{\Var(S_\ell(M;1))}{\Var(S_\ell(M))}\bigg) \bigg)\\\leq &d_W\bigg( \frac{J^*_2(M)}{2\sqrt{\Var(S_\ell(M))}}h_{\ell;2}^*(B), \mathcal{N}\bigg(0, \dfrac{\Var(S_\ell(M;1))}{\Var(S_\ell(M))}\bigg) \bigg)+ d_W\bigg( \frac{J_2(M)h_{\ell;2}(B)}{2\sqrt{\Var(S_\ell(M))}}, \frac{J_2(M)^*h_{\ell;2}^*(B)}{2\sqrt{\Var(S_\ell(M))}}\bigg),
	\end{split}
	\end{equation}
	where $J_2(M)^*$ is the coefficient of the second chaotic component of the chaos expansion of $1_{B;\eps}$.
	In light of (\ref{bound per w}), the latter summand in (\ref{fantasia}) is an $o\bigg(\dfrac{1}{\ell}\bigg)$, thanks to the triangular inequality; whereas, by the Fourth Moment Theorem (see \cite{Peccati Nourdin}, Theorem 5.2.7), the former is
	$O\bigg( \dfrac{1}{\sqrt{\ell}} \bigg)$.  
	For the second term in (\ref{alessio2}), we have that
	\begin{equation}
	\begin{split}
	&d_W\bigg( \frac{J_1(M)}{\sqrt{\Var(S_\ell(M))}}h_{\ell;1}(B)+\frac{J_3(M)}{3!\sqrt{\Var(S_\ell(M))}}h_{\ell;3}(B)+\frac{J_4(M)}{4!\sqrt{\Var(S_\ell(M))}}h_{\ell;4}(B), 0 \bigg) \leq\\& d_W\bigg( \frac{J_1(M)}{\sqrt{\Var(S_\ell(M))}}h_{\ell;1}(B), 0 \bigg)+ d_W\bigg( \frac{J_3(M)}{3!\sqrt{\Var(S_\ell(M))}}h_{\ell;3}(B), 0 \bigg)+d_W\bigg(\frac{J_4(M)}{4!\sqrt{\Var(S_\ell(M))}}h_{\ell;4}(B), 0 \bigg);
	\end{split}
	\end{equation}
	since
	$$d_W\bigg( \frac{J_1(M)}{\sqrt{\Var(S_\ell(M))}}h_{\ell;1}(B), 0 \bigg) \leq \sqrt{\mathbb{E}\bigg[ \bigg(\frac{J_1(M)}{\sqrt{\Var(S_\ell(M))}}h_{\ell;1}(B) \bigg)^2\bigg]}=o \bigg(\dfrac{1}{\sqrt{\ell}}\bigg),$$
	
	$$d_W\bigg( \frac{J_3(M)}{\sqrt{3!\Var(S_\ell(M))}}h_{\ell;3}(B), 0 \bigg) \leq \sqrt{\mathbb{E}\bigg[ \bigg(\frac{J_3(M)}{3!\sqrt{\Var(S_\ell(M))}}h_{\ell;3}(B) \bigg)^2\bigg]}=O \bigg(\sqrt{\dfrac{\sqrt{\log \ell}}{\ell \sqrt{\ell}}}\bigg)$$
	and
	$$d_W\bigg( \frac{J_4(M)}{\sqrt{4!\Var(S_\ell(M))}}h_{\ell;4}(B), 0 \bigg) \leq \sqrt{\mathbb{E}\bigg[ \bigg(\frac{J_4(M)}{4!\sqrt{\Var(S_\ell(M))}}h_{\ell;4}(B) \bigg)^2\bigg]}=O \bigg(\sqrt{\dfrac{\log \ell}{\ell^2}}\bigg),$$
	one has that
	$$d_W\bigg(\dfrac{S_\ell(M;1)}{\sqrt{\Var(S_\ell(M))}}, \mathcal{N}\bigg(0, \dfrac{\Var(S_\ell(M;1))}{\Var(S_\ell(M))}\bigg)\bigg)=O\bigg(\dfrac{1}{\sqrt{\ell}}\bigg).$$
	Finally, Proposition 3.6.1 in \cite{Peccati Nourdin} leads to
	$$d_W\bigg( \mathcal{N}\bigg(0, \dfrac{\Var(S_\ell(M;1))}{\Var(S_\ell(M))}\bigg), \mathcal{N}(0,1) \bigg)\leq \sqrt{\frac{2}{\pi}} \bigg| \dfrac{\Var(S_\ell(M;1))}{\Var(S_\ell(M))}-1\bigg|=O\bigg(\dfrac{1}{\ell}\bigg)$$
	and the thesis of the theorem follows.
\end{proof}

\appendix

\section{Technical details}
In this section we give all the technical details of the proofs of the propositions and the lemmas whereby the main result has been proved.
\subsection{Proof of Lemma \ref{lemma1}}
\begin{proof}
	The aim here is to prove (\ref{smoothvar}). As we have already explained, the ideas is to write the integral in terms of spherical harmonics and then to exploit the properties of the Clebsch-Gordan coefficients. We split the proof in these two steps in order to make the argument clearer.\\\\
	\textit{Step 1}.
	Let us consider the left hand side of (\ref{smoothvar}), we can write it as
	\begin{equation}\label{uu}
	\begin{split}
	\mbox{Var}\bigg(\int_{\mathbb{S}^2}1_{B;\eps}(x) H_2(T_\ell(x)) \,dx \bigg)& =\mathbb{E}\bigg[ \bigg( \int_{\mathbb{S}^2}1_{B;\eps}(x) H_2(T_\ell (x)) \,dx \bigg)^2\bigg]\\&
	=\mathbb{E}\bigg[  \int_{\mathbb{S}^2 \times \mathbb{S}^2}1_{B;\eps}(x) H_2(T_\ell (x)) 1_{B;\eps}(y) H_2(T_\ell (y)) \,dx \,dy\bigg]
	\end{split}
	\end{equation}
	Exchanging the integral and the mean, we have that (\ref{uu}) is
	\begin{equation}\label{uuu}
	\begin{split}
	\\&=\int_{\mathbb{S}^2\times \mathbb{S}^2} 1_{B;\eps}(x)1_{B;\eps}(y) \mathbb{E}[H_2(T_\ell(x)) H_2(T_\ell(y))] dx dy\\&
		\end{split}
	\end{equation}
	and in view of Remark 4.10 in \cite{M e Peccati}, it follows that (\ref{uuu}) is
		\begin{equation}
	\begin{split}
	= 2! \int_{\mathbb{S}^2\times \mathbb{S}^2} 1_{B;\eps}(x)1_{B;\eps}(y)  \mathbb{E}[T_\ell(x) T_\ell(y)]^2 \,dx dy
	=2! \int_{\mathbb{S}^2\times \mathbb{S}^2} 1_{B;\eps}(x)1_{B;\eps}(y)  P_\ell( \langle x,y\rangle)^2 \,dx dy.
	\end{split}
	\end{equation}
	Along the same lines as the proof of (\ref{varianza1}), we replace $P_\ell^2(\langle x,y \rangle)$ with 
	\begin{equation}
	P_\ell(\langle x,y \rangle)^2= \bigg( \dfrac{4\pi}{2\ell+1}\bigg)^2 \sum_{m_1=-\ell}^{\ell} \sum_{m_2=-\ell}^{\ell} Y_{\ell m_1}(x)\overline{Y_{\ell m_1}(y)} \mbox{ }\overline{Y_{\ell m_2}(x)}Y_{\ell m_2}(y)
	\end{equation}
	and (\ref{alessio3}) to obtain

	\begin{equation}
	\begin{split}
	\int_{\mathbb{S}^2 \times \mathbb{S}^2 } & P_\ell(\langle x,y \rangle)^2  1_{B;\eps}(x) 1_{B;\eps}(y) \,dx dy=\\&
	\begin{split}
	= \int_{\mathbb{S}^2 \times \mathbb{S}^2 }  \bigg( \dfrac{4\pi}{2\ell+1}\bigg)^2 \sum_{m_1} \sum_{m_2} Y_{\ell m_1}(x)\overline{Y_{\ell m_2}(x)} \mbox{ } &{Y_{\ell m_2}(y)} \overline{Y_{\ell m_1}(y)}Y_{\ell m_2}(y) \times\\&
	\times\sum_{\ell_1} \sum_{\ell_2} b_{\ell_1;\eps} b_{\ell_2;\eps} Y_{\ell_1 0}(x) Y_{\ell_2 0}(y) dx dy
	\end{split}
	\end{split}
	\end{equation}
	\begin{equation}\label{integral}
	\begin{split}
	=\bigg( \dfrac{4\pi}{2\ell+1}\bigg)^2 \sum_{\ell_1}\sum_{\ell_2} \sum_{m_1} \sum_{m_2}  b_{\ell_1;\eps} b_{\ell_2;\eps} \int_{\mathbb{S}^2}   Y_{\ell m_1}(x) \mbox{ } &Y_{\ell_1 0}(x) \overline{Y_{\ell m_2}(x)} \,dx \times \\&
	\times \int_{\mathbb{S}^2} Y_{\ell m_2}(y)  Y_{\ell_2 0}(y)  \mbox{ } \overline{Y_{\ell m_1}(y)} \,dy;
	\end{split}
	\end{equation}
we already justified the exchange between the series and the integral in Lemma \ref{1chaos}, which follows from (\ref{cond2}).
Now, (\ref{integral}) is known as a Gaunt integral and it is given in \cite{M e Peccati} by the following relation:
	\begin{equation}\label{3.64 libro Domenico}
	\int_{S^2} Y_{\ell_1 m_1} (x) Y_{\ell_2 m_2} (x) \overline{Y}_{\ell_3 m_3}(x) d\sigma(x)= \sqrt{\dfrac{(2\ell_1+1)(2\ell_2+1)}{4\pi (2\ell_3+1)}} C_{\ell_1 m_1 \ell_2 m_2}^{\ell_3 m_3} C_{\ell_1 0 \ell_2 0}^{\ell_3 0},
	\end{equation}
	for all $\ell_1,\ell_2,\ell_3,$
	with the convention that $C_{\ell_1 m_1 \ell_2 m_2}^{\ell_3 m_3}=0$ for those integers $\ell_1,\ell_2,\ell_3$ not satisfying the triangle conditions. 
	Replacing it in (\ref{integral}), one has
	
	\begin{equation}\label{eq}
	\begin{split}
	\bigg( \dfrac{4\pi}{2\ell+1}\bigg)^2 \sum_{\ell_1}\sum_{\ell_2} \sum_{m_1} \sum_{m_2}  b_{\ell_1;\eps}
	\sqrt{\dfrac{(2\ell+1)(2\ell_1+1)}{4\pi(2\ell+1)}} &C_{\ell m_1 \ell_1 0}^{\ell m_2} C_{\ell 0 \ell_1 0}^{\ell 0} \times\\& \times b_{\ell_2;\eps}    \sqrt{\dfrac{(2\ell+1)(2\ell_2+1)}{4\pi(2\ell+1)}} C_{\ell m_2 \ell_2 0}^{\ell m_1} C_{\ell 0 \ell_2 0}^{\ell 0} 
	\end{split}
	\end{equation}
	
	\begin{equation}\label{new}
=\bigg( \dfrac{4\pi}{2\ell+1}\bigg)^2 \dfrac{1}{4\pi} \sum_{\ell_1} b_{\ell_1;\eps} \sqrt{2{\ell_1}+1} C_{\ell 0 \ell_1 0}^{\ell 0} \sum_{\ell_2} b_{\ell_2;\eps}\sqrt{2{\ell_1}+1} C_{\ell 0 \ell_2 0}^{\ell 0} \sum_{m_1 m_2} C_{\ell m_1 \ell_1 0}^{\ell m_2} C_{\ell m_2 \ell_2 0}^{\ell m_1}
\end{equation}
	and then, this is the value of the left hand side in (\ref{smoothvar}).\\\\
	\textit{Step 2}. Now, we just exploit some properties of the Clebsch-Gordan coefficients to simplify the expression in (\ref{new}) and finally to prove that it is equal to the right hand side in (\ref{smoothvar}). Hence, Recalling that the mentioned coefficients are related to the Wigner 3j coefficients by the identities (see \cite{M e Peccati}, Section 3.5.3):
	\begin{equation}\label{3.67 libro Domenico}
	\begin{pmatrix} \ell_1 & \ell_2 & \ell_3 \\ m_1 &
	m_2 & m_3\\ 
	\end{pmatrix}
	= (-1)^{\ell_3+m_3} \dfrac{1}{\sqrt{2\ell_3+1}} C_{\ell_1 -m_1 \ell_2 -m_2}^{\ell_3 m_3}
	\end{equation}
	\begin{equation}\label{3.68 libro Domenico}
	C_{\ell_1 m_1 \ell_2 m_2}^{\ell_3 m_3}=(-1)^{\ell_1-\ell_2+m_3} \sqrt{2\ell_3+1}
	\begin{pmatrix} \ell_1 & \ell_2 & \ell_3 \\ m_1 &
	m_2 & -m_3\\ 
	\end{pmatrix},
	\end{equation}
	and using their permutation property of columns 
	\begin{equation}\label{5 pag 244}
	\begin{pmatrix} \ell_1 & \ell_2 & \ell_3 \\ m_1 &
	m_2 & m_3\\ 
	\end{pmatrix}
	= (-1)^{\ell_1+\ell_2+\ell_3}\begin{pmatrix} \ell_1 & \ell_3 & \ell_2 \\ m_1 &
	m_3 & m_2\\ 
	\end{pmatrix},
	\end{equation}
	it follows that
	\begin{equation}\label{un nome}
	\begin{split}
	C_{\ell m_1 \ell_1 0}^{\ell m_2}&=(-1)^{\ell-\ell_1+m_2} \sqrt{2\ell+1} \begin{pmatrix} \ell & \ell_1 & \ell \\ m_1 & 0 & -m_2 \end{pmatrix}\\&
	=(-1)^{\ell-\ell_1+m_2} \sqrt{2\ell+1} (-1)^{\ell+\ell_1+\ell} \begin{pmatrix} \ell & \ell & \ell_1 \\ m_1 & -m_2 & 0 \end{pmatrix}\\&=
	(-1)^{\ell+m_2+2\ell} \sqrt{2\ell+1} (-1)^{\ell_1+2\ell} \dfrac{1}{\sqrt{2\ell_1+1}} C_{\ell -m_1 \ell m_2}^{\ell_1 0}\\&
	=(-1)^{\ell+m_2+\ell_1} \dfrac{\sqrt{2\ell+1}}{\sqrt{2\ell_1+1}} C_{\ell -m_1 \ell m_2}^{\ell_1 0}
	\end{split}
	\end{equation}
	and
	$$C_{\ell m_2 \ell_2 0}^{\ell m_1}=(-1)^{\ell+m_1+\ell_2} \sqrt{2\ell+1} \dfrac{1}{\sqrt{2\ell_2+1}} C_{\ell -m_2 \ell m_1}^{\ell_2 0};$$
so equation (\ref{eq}) is equal to
	\begin{equation}\label{14}
	\begin{split}
	\bigg( &\dfrac{4\pi}{2\ell+1}\bigg)^2  \dfrac{1}{4\pi} \sum_{\ell_1} b_{\ell_1;\eps} \sqrt{2{\ell_1}+1} C_{\ell 0 \ell_1 0}^{\ell 0} \sum_{\ell_2} b_{\ell_2;\eps}\sqrt{2{\ell_1}+1} C_{\ell 0 \ell_2 0}^{\ell 0} \times\\&
	\sum_{m_1 m_2}  (-1)^{m_1+m_2} (-1)^{\ell_1+\ell_2} \dfrac{\sqrt{2\ell+1}}{\sqrt{2\ell_1+1}} \dfrac{\sqrt{2\ell+1}}{\sqrt{2\ell_2+1}} C_{\ell -m_1 \ell m_2}^{\ell_1 0} C_{\ell -m_2 \ell m_1}^{\ell_2 0} =\\&
	= \dfrac{4\pi}{2\ell+1} \sum_{\ell_1} b_{\ell_1;\eps}  C_{\ell 0 \ell_1 0}^{\ell 0} \sum_{\ell_2} b_{\ell_2;\eps} C_{\ell 0 \ell_2 0}^{\ell 0} (-1)^{\ell_1+\ell_2}\cdot
	\sum_{m_1 m_2}  (-1)^{m_1+m_2}  C_{\ell -m_1 \ell m_2}^{\ell_1 0} C_{\ell -m_2 \ell m_1}^{\ell_2 0}
	\end{split}
	\end{equation}
	and for the triangular condition
	$$m_1-m_2=0 \Rightarrow m_1=m_2$$
	and the unitary relation \cite{quantum theory}:
	
	\begin{equation}\label{unitary prop}
	\sum_{m_1 m_2 } C_{j_1 m_1 j_2 m_2 }^{j m} C_{j_1 m_1 j_2 m_2}^{j' m'}= \delta_{j}^{j'} \delta_{m}^{m'},
	\end{equation}
 (\ref{14}) yields
	
	\begin{equation}\label{eq1}
	\begin{split}
	\dfrac{4\pi}{2\ell+1} \bigg\{ \sum_{\ell_1} (-1)^{\ell_1 } b_{\ell_1;\eps} C_{\ell 0 \ell_1 0}^{\ell 0} \sum_{\ell_2} (-1)^{\ell_2} b_{\ell_2;\eps} C_{\ell 0 \ell_2 0}^{\ell 0} \bigg\} \delta_{\ell_1}^{\ell_2}.
	\end{split}
	\end{equation}
As in (\ref{un nome}) one has
	\begin{equation}
	\begin{split}
	C_{\ell 0 \ell_10}^{\ell 0}&=(-1)^{\ell-\ell_1} \sqrt{2\ell+1}  \begin{pmatrix} \ell & \ell_1 & \ell \\ 0 & 0 & 0 \end{pmatrix} \\&
	=(-1)^{\ell-\ell_1} \sqrt{2\ell+1} \mbox{ }(-1)^{2\ell+\ell_1} \begin{pmatrix} \ell & \ell & \ell_1 \\ 0 & 0 & 0 \end{pmatrix} \\&
	=(-1)^{\ell} \sqrt{2\ell+1} \mbox{ } (-1)^{\ell_1+2\ell} \dfrac{1}{\sqrt{2\ell_1+1}} C_{\ell0 \ell 0 }^{\ell_1 0}\\&
	=(-1)^{\ell+\ell_1} \dfrac{\sqrt{2\ell+1}}{\sqrt{2\ell_1+1}} C_{\ell0 \ell 0 }^{\ell_1 0}
	\end{split}
	\end{equation}
	and then (\ref{eq1}) is 
	\begin{equation}
	\begin{split}
	&\dfrac{4\pi}{2\ell+1} \bigg\{ \sum_{\ell_1} (-1)^{\ell_1} b_{\ell_1;\eps} (-1)^{\ell+\ell_1} \dfrac{\sqrt{2\ell+1}}{\sqrt{2\ell_1+1}}  C_{\ell 0 \ell 0}^{\ell_1 0} \sum_{\ell_2} b_{\ell_2;\eps} (-1)^{\ell_2} (-1)^{\ell+\ell_2} \dfrac{\sqrt{2\ell+1}}{\sqrt{2\ell_2+1}}  C_{\ell 0 \ell 0}^{\ell_2 0} \bigg\} \delta_{\ell_1}^{\ell_2}\\&
	=4\pi \sum_{\ell_1} b_{\ell_1;\eps} \dfrac{1}{\sqrt{2\ell_1+1}} C_{\ell 0 \ell 0}^{\ell_1 0} \sum_{\ell_1} \dfrac{1}{\sqrt{2\ell_1+1}} b_{\ell_1;\eps} C_{\ell 0 \ell 0}^{\ell_1 0}\\&
	=4\pi \bigg\{ \sum_{\ell_1} b_{\ell_1;\eps} \dfrac{1}{\sqrt{2\ell_1+1}} C_{\ell 0 \ell 0}^{\ell_1 0}\bigg\}^2 \\&
	=4\pi \sum_{\ell_1} b_{\ell_1;\eps}^2 \dfrac{1}{2\ell_1+1} \bigg(C_{\ell 0 \ell 0}^{\ell_1 0}\bigg)^2,
	\end{split}
	\end{equation}
	where the last step is due to the previous property (\ref{unitary prop}) with $m_1=m_2=m_3=0$; finally Lemma \ref{lemma1} is proven.
\end{proof}

\subsection{Proof of Lemma \ref{lemma2}}
\begin{proof}
	The variance in (\ref{smoothvar})	
	is bounded from below by a single term of the series in the right hand side of (\ref{smoothvar}), i.e.,
	$$8 \pi b_{\bar{\ell_1};\eps}^2 \dfrac{1}{2\bar{\ell_1}+1} \bigg(C_{\ell 0 \ell 0}^{\bar{\ell_1} 0}\bigg)^2,$$
	for a fixed $\bar{\ell_1}$ of the sum; for instance  $\bar{\ell}_1=0$, i.e.,
	$$\mbox{Var}\bigg[\int_{S^2} 1_{B;\eps}(x) H_2(T_\ell(x)) \,dx \bigg]=8 \pi \sum_{\ell_1} b_{\ell_1;\eps}^2 \dfrac{1}{2\ell_1+1} \bigg(C_{\ell 0 \ell 0}^{\ell_1 0}\bigg)^2 \geq 8 \pi b_{0;\eps}^2  \bigg(C_{\ell 0 \ell 0}^{0 0}\bigg)^2=8\pi b_{0;\eps}^2 \dfrac{1}{2\ell+1},$$
by the property
	\begin{equation}\label{1 pag 248}
	C_{\ell_1 m_1 \ell_2 m_2}^{ 00} =(-1)^{\ell_1-m_1} \dfrac{\delta_{\ell_1}^{\ell_2} \delta_{m_1}^{-m_2}}{\sqrt{2\ell_1+1}}
	\end{equation}
	(see \cite{quantum theory}).
	To find an upper bound, it is sufficient to recall that for any $\ell_1,\ell_2,\ell_3,$ 
	\begin{equation}
	\bigg| \begin{pmatrix} \ell_1 & \ell_2 & \ell_3 \\ m_1 &
	m_2 & m_3\\ 
	\end{pmatrix}
	\bigg| \leq [\max\{2\ell_1+1,2\ell_2+1,2\ell_3+1\}]^{-1/2}
	\end{equation}
	(see \cite{M e Peccati} p. 110) so that
	\begin{equation}\label{mod c}
	|C_{\ell_1 m_1 \ell_2 m_2}^{\ell_3 m_3}|\leq \sqrt{2\ell_3+1} [\max\{2\ell_1+1,2\ell_2+1,2\ell_3+1\}]^{-1/2}
	\end{equation}
and 
	then, it is easily seen that 
	$$8 \pi \sum_{\ell_1} b_{\ell_1;\eps}^2 \dfrac{1}{2\ell_1+1} \bigg(C_{\ell 0 \ell 0}^{\ell_1 0}\bigg)^2 \leq 8\pi \sum_{\ell_1} b_{\ell_1;\eps}^2 \dfrac{1}{2\ell_1+1} \dfrac{2\ell_1+1}{2\ell+1}= \dfrac{8\pi}{2\ell+1} \sum_{\ell_1} b_{\ell_1;\eps}^2.$$
	The series is finite by Remark \ref{norma}.
	In conclusion, (\ref{smoothvar}) is bounded above and below by
	\begin{equation}
	\begin{split}
	8\pi b_{0;\eps}^2 \dfrac{1}{2\ell+1} \leq 8 \pi \sum_{\ell_1} b_{\ell_1;\eps}^2 \dfrac{1}{2\ell_1+1} \bigg(C_{\ell 0 \ell 0}^{\ell_1 0}\bigg)^2 \leq \dfrac{8\pi}{2\ell+1} \sum_{\ell_1} b_{\ell_1;\eps}^2 \leq \dfrac{8\pi}{2\ell+1} m(\mathbb{S}^2)= \dfrac{8\pi}{2\ell+1} 4\pi
	\end{split}
	\end{equation}
	and in light of Remark \ref{remark b0}, the lemma is proved.
\end{proof}

\subsection{Proof of Proposition \protect\ref{4 cumulant}}

\begin{proof}
	The purpose here is to compute the fourth cumulant of $h_{\ell;2}^*(B)$, which, in view of the Diagram Formula (see Section 4.3 \cite{M e Peccati}, or \cite{M e W 2012}), is given by
	\begin{equation}  \label{cum4}
	\begin{split}
	cum_4(h^*_{\ell;2}(B))=\int_{\mathbb{S}^2}1_{B;\varepsilon}(x)&\int_{\mathbb{S}^2}1_{B;%
		\varepsilon}(z)\int_{\mathbb{S}^2} P_\ell(\langle x,y\rangle ) P_\ell(\langle
	y,z\rangle) 1_{B;\varepsilon}(y) \,dy  \\
	& \times \int_{\mathbb{S}^2} P_\ell(\langle z,w\rangle)P_\ell(\langle w,x\rangle)
	1_{B;\varepsilon}(w) \,dw \,dx \,dz;
	\end{split}%
	\end{equation}
	the idea is always to write the integral in terms of spherical harmonics and hence, of the Clebsch-Gordan coefficients. Then, the second step is to handle them in order to derive an expression with less parameters, namely equation (\ref{36}). In the third step we split the series given in (\ref{36}) to make neater some terms of the series. At this point, in step 4, we study the asymptotic behavior of all these terms proving, finally, the thesis of the lemma.\\\\
	\textit{Step 1}. Putting together (\ref{add formula}) and (\ref{fn caratteristica}) in (\ref{cum4}), we obtain four Gaunt integrals and  (\ref{3.64 libro Domenico}) implies that (\ref{cum4}) is equal to
	
	
	\begin{equation}\label{a22}
	\begin{split}
	\bigg(\frac{4\pi}{2\ell+1}\bigg)^4 \dfrac{1}{(4\pi)^2} &
	\sum_{\ell_1=-\ell}^{\ell} b_{\ell_1;\varepsilon} \sqrt{(2\ell_1+1)} C_{\ell
		0 \ell_1 0}^{\ell 0} \sum_{\ell_2=0}^{\infty} b_{\ell_2;\varepsilon} \sqrt{%
		(2\ell_2+1)} C_{\ell 0 \ell_2 0}^{\ell 0} \\
	& \sum_{\ell_3=-\ell}^{\ell} b_{\ell_3;\varepsilon} \sqrt{(2\ell_3+1)}
	C_{\ell 0 \ell_3 0}^{\ell 0} \sum_{\ell_4=-\ell}^{\ell}
	b_{\ell_4;\varepsilon} \sqrt{(2\ell_4+1)} C_{\ell 0 \ell_4 0}^{\ell 0} \\
	& \sum_{m_1=-\ell}^{\ell} \sum_{m_2=-\ell}^{\ell} \sum_{m_3=-\ell}^{\ell}
	\sum_{m_4=-\ell}^{\ell} C_{\ell m_1 \ell_10}^{\ell m_2}C_{\ell m_3
		\ell_20}^{\ell m_4} C_{\ell m_4 \ell_30}^{\ell m_1}C_{\ell m_2
		\ell_40}^{\ell m_3}.
	\end{split}%
	\end{equation}
	
	\textit{Step 2}. In this step we reduce the number of parameters. Indeed, the triangular condition implies that $m_1=m_2=m_3=m_4$, so that (\ref{a22}) is
	\begin{equation}  \label{33}
	\begin{split}
	= \bigg(\frac{4\pi}{2\ell+1}\bigg)^4 \dfrac{1}{(4\pi)^2} &
	\sum_{\ell_1=-\ell}^{\ell} b_{\ell_1;\varepsilon} \sqrt{(2\ell_1+1)} C_{\ell
		0 \ell_1 0}^{\ell 0} \sum_{\ell_2=0}^{\infty} b_{\ell_2;\varepsilon} \sqrt{%
		(2\ell_2+1)} C_{\ell 0 \ell_2 0}^{\ell 0} \\
	& \sum_{\ell_3=-\ell}^{\ell} b_{\ell_3;\varepsilon} \sqrt{(2\ell_3+1)}
	C_{\ell 0 \ell_3 0}^{\ell 0} \sum_{\ell_4=-\ell}^{\ell}
	b_{\ell_4;\varepsilon} \sqrt{(2\ell_4+1)} C_{\ell 0 \ell_4 0}^{\ell 0} \\
	& \sum_{m_1=-\ell}^{\ell} C_{\ell m_1 \ell_10}^{\ell m_1}C_{\ell m_1
		\ell_20}^{\ell m_1} C_{\ell m_1 \ell_30}^{\ell m_1}C_{\ell m_1
		\ell_40}^{\ell m_1}.
	\end{split}%
	\end{equation}
	Besides, for the symmetry properties (\ref{symmetry p 10} or \cite{quantum theory}), one has that
	\begin{equation*}
	C_{\ell m_1 \ell_1 0}^{\ell m_1}=(-1)^{\ell-m_1} \sqrt{\dfrac{2\ell+1}{%
			2\ell_1+1}} C_{\ell m_1 \ell -m_1}^{\ell_1 0},
	\end{equation*}
	and then
	\begin{equation}\label{newa22}
	\begin{split}
	\sum_{m_1=-\ell}^{\ell} &C_{\ell m_1 \ell_10}^{\ell m_1}C_{\ell m_1
		\ell_20}^{\ell m_1} C_{\ell m_1 \ell_30}^{\ell m_1}C_{\ell m_1
		\ell_40}^{\ell m_1} = \\
	& =\sum_{m_1=-\ell}^{\ell} (-1)^{4(\ell-m_1)} \sqrt{\dfrac{(2\ell+1)^4}{%
			(2\ell_1+1)(2\ell_2+1)(2\ell_3+1)(2\ell_4+1)}} C_{\ell m_1 \ell m_1}^{\ell_1
		0}C_{\ell m_1 \ell m_1}^{\ell_2 0} C_{\ell m_1 \ell m_1}^{\ell_3 0}C_{\ell
		m_1 \ell m_1}^{\ell_4 0} .
	\end{split}%
	\end{equation}
	By (\ref{2 pag 248}) and
	(\ref{20 p.260}) (\cite{quantum theory}), (\ref{newa22}) becomes
	\begin{equation}\label{a28}
	\begin{split}
	= \bigg(\frac{4\pi}{2\ell+1}\bigg)^2 & \sum_{\ell_1=-\ell}^{\ell}
	b_{\ell_1;\varepsilon} C_{\ell 0 \ell_1 0}^{\ell 0} \sum_{\ell_2=0}^{\infty}
	b_{\ell_2;\varepsilon} C_{\ell 0 \ell_2 0}^{\ell 0} \times\\
	& \sum_{\ell_3=-\ell}^{\ell} b_{\ell_3;\varepsilon} C_{\ell 0 \ell_3
		0}^{\ell 0} \sum_{\ell_4=-\ell}^{\ell} b_{\ell_4;\varepsilon} C_{\ell 0
		\ell_4 0}^{\ell 0} \times\\
	& \prod_{\ell_1,\ell_2,\ell_3,\ell_4} \sum_{k j} C_{\ell_3 0 \ell_4 0}^{k
		j}C_{\ell_2 0 \ell_1 0}^{k j} 
	\begin{Bmatrix}
	\ell & \ell & \ell_1 \\ 
	\ell & \ell & \ell_2 \\ 
	\ell_4 & \ell_3 & k%
	\end{Bmatrix}%
	,
	\end{split}%
	\end{equation}
	where the last symbol is the Wigner 9j coefficient (see \cite{quantum theory}).
	Then, since the triangular condition implies $j=0$, (\ref{a28}) gives
	\begin{equation}\label{a29}
	\begin{split}
	= \bigg(\frac{4\pi}{2\ell+1}\bigg)^2 & \sum_{\ell_1=-\ell}^{\ell}
	b_{\ell_1;\varepsilon} C_{\ell 0 \ell_1 0}^{\ell 0} \sum_{\ell_2=0}^{\infty}
	b_{\ell_2;\varepsilon} C_{\ell 0 \ell_2 0}^{\ell 0} \times\\
	& \sum_{\ell_3=-\ell}^{\ell} b_{\ell_3;\varepsilon} C_{\ell 0 \ell_3
		0}^{\ell 0} \sum_{\ell_4=-\ell}^{\ell} b_{\ell_4;\varepsilon} C_{\ell 0
		\ell_4 0}^{\ell 0} \times\\
	& \prod_{\ell_1 \ell_2 \ell_3 \ell_4 }\sum_{k} C_{\ell_3 0 \ell_4 0}^{k
		0}C_{\ell_2 0 \ell_1 0}^{k 0} 
	\begin{Bmatrix}
	\ell & \ell & \ell_1 \\ 
	\ell & \ell & \ell_2 \\ 
	\ell_4 & \ell_3 & k%
	\end{Bmatrix}.%
	\end{split}%
	\end{equation}
	In view of equation (\ref{symmetry p 10}), (\ref{a29}) reduces to
	\begin{equation}  \label{36}
	\begin{split}
	= (4\pi)^2 & \sum_{\ell_1=0}^{\infty} b_{\ell_1;\varepsilon} C_{\ell 0 \ell
		0}^{\ell_1 0} \sum_{\ell_2=0}^{\infty} b_{\ell_2;\varepsilon} C_{\ell 0 \ell
		0}^{\ell_2 0} \sum_{\ell_3=0}^{\infty} b_{\ell_3;\varepsilon} C_{\ell 0 \ell
		0}^{\ell_3 0} \sum_{\ell_4=0}^{\infty} b_{\ell_4;\varepsilon} C_{\ell 0 \ell
		0}^{\ell_4 0} \sum_{k} C_{\ell_3 0 \ell_4 0}^{k 0}C_{\ell_2 0 \ell_1 0}^{k
		0} 
	\begin{Bmatrix}
	\ell & \ell & \ell_1 \\ 
	\ell & \ell & \ell_2 \\ 
	\ell_4 & \ell_3 & k%
	\end{Bmatrix}%
	.
	\end{split}%
	\end{equation}
	
	\textit{Step 3}.
	In order to simplify the notation, we define this last expression as $A_{\ell,k}(\ell_1,\ell_2,\ell_3,\ell_4)$. We split it in different cases and we study them separately;
	hence, we rewrite (\ref{36}) as
	\begin{equation}  \label{quarto cum}
	\begin{split}
	= &A_{\ell,k}(\ell_{1},\ell_2,\ell_3,\ell_4)+ A_{\ell,k}(0,0,0,0)+
	A_{\ell,k}(0,\ell_2,\ell_3,\ell_4)+A_{\ell,k}(\ell_1,0,\ell_3,\ell_4)+A_{%
		\ell,k}(\ell_1,\ell_2,0,\ell_4)+ \\
	&A_{\ell,k}(\ell_1,\ell_2,\ell_3,0)+
	2A_{\ell,k}(0,0,\ell_3,\ell_4)+2A_{\ell,k}(0,\ell_2,0,\ell_4)+2A_{\ell,k}(0,%
	\ell_2,\ell_3,0)+2A_{\ell,k}(\ell_1,0,0,\ell_4)+ \\
	&2A_{\ell,k}(\ell_1,0,\ell_3,0)+2A_{\ell,k}(\ell_1,\ell_2,0,0),
	\end{split}
	\end{equation}
	where
	\begin{equation}
	\begin{split}
	A_{\ell,k}&(\ell_1,\ell_2,\ell_3,\ell_4):= \\
	&(4\pi)^2 \sum_{\ell_1=1}^{\infty} b_{\ell_1;\varepsilon} C_{\ell 0 \ell
		0}^{\ell_1 0} \sum_{\ell_2=1}^{\infty} b_{\ell_2;\varepsilon} C_{\ell 0 \ell
		0}^{\ell_2 0} \sum_{\ell_3=1}^{\infty} b_{\ell_3;\varepsilon} C_{\ell 0 \ell
		0}^{\ell_3 0} \sum_{\ell_4=1}^{\infty} b_{\ell_4;\varepsilon} C_{\ell 0 \ell
		0}^{\ell_4 0} \sum_{k} C_{\ell_3 0 \ell_4 0}^{k 0}C_{\ell_2 0 \ell_1 0}^{k
		0} 
	\begin{Bmatrix}
	\ell & \ell & \ell_1 \\ 
	\ell & \ell & \ell_2 \\ 
	\ell_4 & \ell_3 & k%
	\end{Bmatrix},
	\end{split}
	\end{equation}
	and so that
	\begin{equation*}
	A_{\ell,k}(0,0,0,0)=+(4\pi)^2 b_{0;\varepsilon}^4 (C_{\ell 0 \ell 0}^{00})^4
	C_{0000}^{00}C_{0000}^{00}%
	\begin{Bmatrix}
	\ell & \ell & 0 \\ 
	\ell & \ell & 0 \\ 
	0 & 0 & 0%
	\end{Bmatrix},
	\end{equation*}
	\begin{equation*}
	A_{\ell,k}(0,\ell_2,\ell_3,\ell_4)=(4\pi)^2 b_{0;\varepsilon} C_{\ell 0 \ell
		0}^{00} \sum_{\ell_2,\ell_3, \ell_4=1}^{\infty} b_{\ell_2;\varepsilon}
	C_{\ell 0 \ell 0}^{\ell_2 0} b_{\ell_3;\varepsilon} C_{\ell 0 \ell
		0}^{\ell_3 0} b_{\ell_4;\varepsilon} C_{\ell 0 \ell 0}^{\ell_4 0} \sum_{k}
	C_{\ell_3 0 \ell_4 0}^{k 0}C_{\ell_2 0 0 0}^{k 0} 
	\begin{Bmatrix}
	\ell & \ell & 0 \\ 
	\ell & \ell & \ell_2 \\ 
	\ell_4 & \ell_3 & k%
	\end{Bmatrix}%
	\end{equation*}
	(and similar expressions hold for $A_{\ell,k}(\ell_1,0,\ell_3,\ell_4),\mbox{ } A_{\ell,k}(\ell_1,\ell_2,0,\ell_4),\mbox{ } A_{\ell,k}(\ell_1,\ell_2,\ell_3,0)$),
	\begin{equation*}
	\begin{split}
	A_{\ell,k}(0,0,\ell_3,\ell_4)=(4\pi)^2 (b_{0,\varepsilon} C_{\ell 0 \ell
		0}^{00})^2 \sum_{\ell_3, \ell_2=1}^{\infty} b_{\ell_3;\varepsilon} &C_{\ell 0
		\ell 0}^{\ell_3 0} b_{\ell_2;\varepsilon} C_{\ell 0 \ell 0}^{\ell_2 0}\times\\& \times
	\sum_{k} C_{\ell_3 0 0 0}^{k 0}C_{\ell_2 0 0 0}^{k 0} 
	\begin{Bmatrix}
	\ell & \ell & 0 \\ 
	\ell & \ell & \ell_2 \\ 
	0 & \ell_3 & k%
	\end{Bmatrix}.
	\end{split}
	\end{equation*}
	Note that all the terms with three indexes among $\ell_1,\ell_2,\ell_3,\ell_4$
	equal to zero, are zero for the triangular condition, in fact, if we look at the term 
	\begin{equation*}
	3 (4\pi)^2 (b_{0,\varepsilon} C_{\ell 0 \ell 0}^{00})^3
	\sum_{\ell_3=1}^{\infty} b_{\ell_3;\varepsilon} C_{\ell 0 \ell 0}^{\ell_3 0}
	\sum_{k} C_{\ell_3 0 0 0}^{k 0}C_{0 0 0 0}^{k 0} 
	\begin{Bmatrix}
	\ell & \ell & 0 \\ 
	\ell & \ell & 0 \\ 
	0 & \ell_3 & k%
	\end{Bmatrix}%
	,
	\end{equation*}
	in the last sum, the Clebsch-Gordan coefficient $C_{\ell_30 00}^{k0}$ is
	different from zero only if $\ell_3=0$, but this value of $\ell_3$ is not considered in the
	current series.\\\\
	Now we simplify a bit these expressions.
	
	\begin{itemize}

		\item[-] As far as $A_{\ell,k}(0,0,0,0)$ is concerned, for the symmetry properties of the 9j symbols \cite{quantum
			theory}, 
		one has 
		\begin{equation*}
		\begin{Bmatrix}
		\ell & \ell & 0 \\ 
		\ell & \ell & 0 \\ 
		0 & 0 & 0%
		\end{Bmatrix}%
		=\dfrac{1}{2\ell+1}.
		\end{equation*} Therefore, remembering that $C_{0000}^{00}=1$ (from (\ref{2 pag 248})) and $%
		C_{\ell 0 \ell 0}^{0 0}=\dfrac{(-1)^\ell}{\sqrt{2\ell+1}}$ (from (\ref{1 pag
			248})), 
		\begin{equation}\label{stella}
		A_{\ell,k}(0,0,0,0)=(4\pi)^2 b_{0,\varepsilon}^4 \dfrac{1}{(2\ell+1)^3}.
		\end{equation}

		\item[-] Look at the term $A_{\ell;k}(0,\ell_2,\ell_3,\ell_4)$; for the
		triangular condition the only term in the sum in $k$ which does not vanish is $%
		k=\ell_2$ and for the symmetry properties of the Wigner 9j coefficients (\ref{2 p.357}) and
		for (\ref{1 p.357}),
		it follows that
		\begin{equation*}
		\begin{Bmatrix}
		\ell & \ell & 0 \\ 
		\ell & \ell & \ell_2 \\ 
		\ell_4 & \ell_3 & \ell_2%
		\end{Bmatrix}%
		=%
		\begin{Bmatrix}
		\ell_3 & \ell_4 & \ell_2 \\ 
		\ell & \ell & \ell_2 \\ 
		\ell & \ell & 0%
		\end{Bmatrix}%
		=\dfrac{(-1)^{\ell_4+\ell_2}}{[(2\ell_2+1)(2\ell+1)]^{1/2}} 
		\begin{Bmatrix}
		\ell_3 & \ell_4 & \ell_2 \\ 
		\ell & \ell & \ell%
		\end{Bmatrix}%
		.
		\end{equation*}
		Likewise,
		\begin{equation*}
		\begin{Bmatrix}
		\ell & \ell & \ell_1 \\ 
		\ell & \ell & 0 \\ 
		\ell_4 & \ell_3 & \ell_1%
		\end{Bmatrix}%
		=%
		\begin{Bmatrix}
		\ell & \ell & \ell_1 \\ 
		\ell_3 & \ell_4 & \ell_1 \\ 
		\ell & \ell & 0%
		\end{Bmatrix}
		=\dfrac{(-1)^{\ell_3+\ell_1}}{[(2\ell_1+1)(2\ell+1)]^{1/2}} 
		\begin{Bmatrix}
		\ell & \ell & \ell_1 \\ 
		\ell_4 & \ell_3 & \ell%
		\end{Bmatrix}%
		=
		\end{equation*}
		\begin{equation*}
		=\dfrac{(-1)^{\ell_3+\ell_1}}{[(2\ell_1+1)(2\ell+1)]^{1/2}} 
		\begin{Bmatrix}
		\ell_3 & \ell_4 & \ell_1 \\ 
		\ell & \ell & \ell%
		\end{Bmatrix}%
		,
		\end{equation*}
		where the last equality is due to the invariance under permutation of the Wigner
		6j coefficients (\ref{2 p.298}). 
		Similarly, 
		\begin{equation*}
		\begin{Bmatrix}
		\ell & \ell & \ell_1 \\ 
		\ell & \ell & \ell_2 \\ 
		\ell_4 & 0 & \ell_4%
		\end{Bmatrix}%
		=%
		\begin{Bmatrix}
		\ell & \ell_2 & \ell \\ 
		\ell & \ell_1 & \ell \\ 
		\ell_4 & \ell_4 & 0%
		\end{Bmatrix}%
		=\dfrac{(-1)^{\ell_2+\ell_4}}{[(2\ell_4+1)(2\ell+1)]^{1/2}} 
		\begin{Bmatrix}
		\ell & \ell_2 & \ell \\ 
		\ell_1 & \ell & \ell_4%
		\end{Bmatrix}%
		=
		\end{equation*}
		\begin{equation*}
		=\dfrac{(-1)^{\ell_2+\ell_4}}{[(2\ell_4+1)(2\ell+1)]^{1/2}} 
		\begin{Bmatrix}
		\ell_1 & \ell_2 & \ell_4 \\ 
		\ell & \ell & \ell%
		\end{Bmatrix}%
		\end{equation*}
		and 
		\begin{equation*}
		\begin{Bmatrix}
		\ell & \ell & \ell_1 \\ 
		\ell & \ell & \ell_2 \\ 
		0 & \ell_3 & \ell_3%
		\end{Bmatrix}%
		=%
		\begin{Bmatrix}
		\ell & \ell_2 & \ell \\ 
		\ell & \ell_1 & \ell \\ 
		\ell_4 & \ell_4 & 0%
		\end{Bmatrix}%
		=\dfrac{(-1)^{\ell_1+\ell_3}}{[(2\ell_3+1)(2\ell+1)]^{1/2}} 
		\begin{Bmatrix}
		\ell_2 & \ell & \ell \\ 
		\ell & \ell_1 & \ell_3%
		\end{Bmatrix}%
		\end{equation*}
		\begin{equation*}
		=\dfrac{(-1)^{\ell_1+\ell_3}}{[(2\ell_3+1)(2\ell+1)]^{1/2}} 
		\begin{Bmatrix}
		\ell_2 & \ell_1 & \ell_3 \\ 
		\ell & \ell & \ell%
		\end{Bmatrix}%
		.
		\end{equation*}
		
		Renaming the indexes of the similar terms with $\ell_2, \ell_3$, we can write
		\begin{equation}\label{key}
		\begin{split}
		&A_{\ell,k}(0,\ell_2,\ell_3,\ell_4)=A_{\ell,k}(\ell_2,0,\ell_3,\ell_4)=A_{%
			\ell,k}(\ell_2,\ell_3,0,\ell_4)=A_{\ell,k}(\ell_2,\ell_3,\ell_4,0)= \\
		& =(4\pi)^2 b_{0,\varepsilon} \dfrac{(-1)^{\ell}}{\sqrt{2\ell+1}}
		\sum_{\ell_2,\ell_3, \ell_4=1}^{\infty} b_{\ell_2;\varepsilon} C_{\ell 0
			\ell 0}^{\ell_2 0} b_{\ell_3;\varepsilon} C_{\ell 0 \ell 0}^{\ell_3 0}
		b_{\ell_4;\varepsilon} C_{\ell 0 \ell 0}^{\ell_4 0} C_{\ell_3 0 \ell_4
			0}^{\ell_2 0} \times \\& \times\dfrac{(-1)^{\ell_2+\ell_4}}{\sqrt{(2\ell+1)(2\ell_2+1)}}%
		\begin{Bmatrix}
		\ell_3 & \ell_4 & \ell_2 \\ 
		\ell & \ell & \ell%
		\end{Bmatrix}%
		= \\
		& =(4\pi)^2 b_{0,\varepsilon} \dfrac{(-1)^{\ell}}{2\ell+1}
		\sum_{\ell_2,\ell_3, \ell_4=1}^{\infty} b_{\ell_2;\varepsilon} C_{\ell 0
			\ell 0}^{\ell_2 0} b_{\ell_3;\varepsilon} C_{\ell 0 \ell 0}^{\ell_3 0}
		b_{\ell_4;\varepsilon} C_{\ell 0 \ell 0}^{\ell_4 0} C_{\ell_3 0 \ell_4
			0}^{\ell_2 0} \dfrac{(-1)^{\ell_2+\ell_4}}{\sqrt{(2\ell_2+1)}}%
		\begin{Bmatrix}
		\ell_3 & \ell_4 & \ell_2 \\ 
		\ell & \ell & \ell%
		\end{Bmatrix}%
		\end{split}%
		\end{equation}
		and since $C_{\ell0\ell0}^{\ell^{\prime }0}=0$ if $\ell^{\prime }$ is odd, the
		series only run on even indexes and this implies that $(-1)^{\ell_2+\ell_4}=1$.
		
		\item[-] Regarding $A_{\ell,k}(0,0,\ell_3,\ell_4)$, for the triangular condition,
		the only term of the sum in $k$ which is non-zero is $k=0$ and the symmetry
		properties of the 9j symbol (\ref{p.343}) and the relation (\ref{4 p.358})
		imply 
		\begin{equation}\label{permutation9j}
		\begin{Bmatrix}
		\ell & \ell & 0 \\ 
		\ell & \ell & 0 \\ 
		\ell_4 & \ell_3 & 0%
		\end{Bmatrix}%
		=%
		\begin{Bmatrix}
		\ell & \ell & \ell_4 \\ 
		\ell & \ell & \ell_3 \\ 
		0 & 0 & 0%
		\end{Bmatrix}%
		=\dfrac{\delta_{\ell_3}^{\ell_4}}{[(2\ell_4+1)(2\ell+1)^2]^{1/2}}.
		\end{equation}
		The same properties give
		\begin{equation*}
		\begin{Bmatrix}
		\ell & \ell & 0 \\ 
		\ell & \ell & \ell_2 \\ 
		\ell_2 & 0 & \ell_2%
		\end{Bmatrix}%
		=%
		\begin{Bmatrix}
		\ell & \ell_2 & \ell \\ 
		\ell & 0 & \ell \\ 
		\ell_2 & \ell_2 & 0%
		\end{Bmatrix}%
		=\dfrac{1}{[(2\ell_2+1)(2\ell+1)]^{1/2}} 
		\begin{Bmatrix}
		\ell & \ell_2 & \ell \\ 
		0 & \ell & \ell_2%
		\end{Bmatrix}%
		\end{equation*}
		and since if one of the argument is zero the value of the 6j symbol can be
		written explicitly as in (\ref{1 pag.299}),
		we have that
		\begin{equation*}
		\begin{Bmatrix}
		\ell & \ell & 0 \\ 
		\ell & \ell & \ell_2 \\ 
		\ell_2 & 0 & \ell_2%
		\end{Bmatrix}%
		=\dfrac{1}{[(2\ell_2+1)(2\ell+1)]^{1/2}} \dfrac{(-1)^{\ell_2}}{\sqrt{%
				(2\ell_2+1)(2\ell+1)}} =\dfrac{(-1)^{\ell_2}}{(2\ell_2+1)(2\ell+1)}.
		\end{equation*}
		Analogously, we get
		\begin{equation*}
		\begin{Bmatrix}
		\ell & \ell & 0 \\ 
		\ell & \ell & \ell_2 \\ 
		0 & \ell_2 & \ell_2%
		\end{Bmatrix}%
		= \dfrac{(-1)^{\ell_2}}{[(2\ell_2+1)(2\ell+1)]^{1/2}}%
		\begin{Bmatrix}
		\ell_2 & 0 & \ell_2 \\ 
		\ell & \ell & \ell \\ 
		\end{Bmatrix}
		=\dfrac{1}{(2\ell_2+1)(2\ell+1)},
		\end{equation*}
		\begin{equation*}
		\begin{Bmatrix}
		\ell & \ell & \ell_1 \\ 
		\ell & \ell & 0 \\ 
		\ell_1 & 0 & \ell_1%
		\end{Bmatrix}%
		= \dfrac{(-1)^{\ell_1}}{[(2\ell_1+1)(2\ell+1)]^{1/2}}%
		\begin{Bmatrix}
		\ell & \ell & \ell_1 \\ 
		\ell_1 & 0 & \ell \\ 
		\end{Bmatrix}
		=\dfrac{1}{(2\ell_1+1)(2\ell+1)},
		\end{equation*}
		\begin{equation*}
		\begin{Bmatrix}
		\ell & \ell & \ell_1 \\ 
		\ell & \ell & 0 \\ 
		0 & \ell_1 & \ell_1%
		\end{Bmatrix}%
		= \dfrac{1}{[(2\ell_1+1)(2\ell+1)]^{1/2}}%
		\begin{Bmatrix}
		\ell & \ell & \ell_1 \\ 
		0 & \ell_1 & \ell \\ 
		\end{Bmatrix}
		=\dfrac{(-1)^{\ell_1}}{(2\ell_1+1)(2\ell+1)}
		\end{equation*}
		and 
		\begin{equation*}
		\begin{Bmatrix}
		\ell & \ell & \ell_1 \\ 
		\ell & \ell & \ell_1 \\ 
		0 & 0 & 0%
		\end{Bmatrix}%
		=\dfrac{1}{[(2\ell_1+1)(2\ell+1)^2]^{1/2}}.
		\end{equation*}
		Same computations of (\ref{key}) lead to
		\begin{equation}
		\begin{split}
		&A_{\ell,k}(0,0,\ell_3,\ell_4)=A_{\ell,k}(0,\ell_3,0,\ell_4)
		=A_{\ell,k}(\ell_3,0,\ell_4,0)=A_{\ell,k}(\ell_3,\ell_4,0,0)= \\
		& =(4\pi)^2 \dfrac{b_{0,\varepsilon}^2}{2\ell+1} \sum_{\ell_3=1}^{\infty}
		b_{\ell_3;\varepsilon}^2 (C_{\ell 0 \ell 0}^{\ell_3 0})^2 \dfrac{%
			(-1)^{\ell_3}}{{(2\ell+1)(2\ell_3+1)}}\\&=(4\pi)^2 \dfrac{b_{0,\varepsilon}^2}{%
			(2\ell+1)^2} \sum_{\ell_3=1}^{\infty} b_{\ell_3;\varepsilon}^2 (C_{\ell 0
			\ell 0}^{\ell_3 0})^2 \dfrac{(-1)^{\ell_3}}{{(2\ell_3+1)}}
		\end{split}%
		\end{equation}
		and 
		\begin{equation*}
		A_{\ell,k}(0,\ell_2,\ell_3,0)=A_{\ell,k}(\ell_1,0,0,\ell_4)=(4\pi)^2 \dfrac{%
			b_{0,\varepsilon}^2}{(2\ell+1)^2} \sum_{\ell_3=1}^{\infty}
		b_{\ell_3;\varepsilon}^2 (C_{\ell 0 \ell 0}^{\ell_3 0})^2 \dfrac{1}{{%
				(2\ell_3+1)}};
		\end{equation*}
		as for the previous case, since $C_{\ell0\ell0}^{\ell^{\prime }0}=0$ if $\ell^{\prime }$ is odd, the
		series only run on even indexes and then $%
		(-1)^{\ell_3}=1$. 
	\end{itemize}
	Finally, equation (\ref{quarto cum}) reduces to
	\begin{equation}  \label{sum}
	\begin{split}
	cum_4&(h^*_{\ell;2})=A_{\ell,k}(\ell_{1},\ell_2,\ell_3,\ell_4)+
	A_{\ell,k}(0,0,0,0)+ 4A_{\ell,k}(0,\ell_2,\ell_3,\ell_4)+
	12A_{\ell,k}(0,0,\ell_3,\ell_4).
	\end{split}%
	\end{equation}
	\textit{Step 4}. The aim now is to understand the asymptotic behavior of expression (\ref{sum}) and to prove that it is $O\bigg(\dfrac{1}{\ell^3}\bigg)$. As in the previous step we study the summand in (\ref{sum}) separately.
	
	\begin{itemize}
		\item[-] First, we can note that, by the results of the second chaotic component, it is easily seen that the last summand of (\ref{sum}) is $O\bigg(\dfrac{1}{\ell^3}\bigg)$, as $\ell \rightarrow \infty$. The same holds for $A_{\ell,k}(0,0,0,0)$ directly from (\ref{stella}).
		\item[-] Concerning the third term of (\ref{sum}), because of (\ref{1 p110}) (see \cite{M e Peccati}), the following upper bound holds
		\begin{equation}\label{nonsaprei3}
		\bigg| 
		\begin{Bmatrix}
		\ell_3 & \ell_4 & \ell_2 \\ 
		\ell & \ell & \ell%
		\end{Bmatrix}
		\bigg| \leq \dfrac{1}{\sqrt{2\ell+1}} \min \bigg(\dfrac{1}{\sqrt{2\ell_2+1}}%
		, \dfrac{1}{\sqrt{2\ell_3+1}}, \dfrac{1}{\sqrt{2\ell_4+1}}\bigg)
		\end{equation}
		and from (\ref{mod c}), 
		\begin{equation}  \label{star0}
		|C_{\ell 0 \ell 0}^{\ell_2 0}| \leq \dfrac{\sqrt{2\ell_2+1}}{\sqrt{2\ell+1}},
		\end{equation}
		taking the absolute value, one has that $A_{\ell,k}(0,\ell_2,\ell_3,\ell_4)$ is bounded by
		\begin{equation}\label{third}
		\begin{split}
		&\leq (4\pi)^2 b_{0,\varepsilon} \dfrac{%
			(-1)^{\ell}}{2\ell+1} \sum_{\ell_2,\ell_3, \ell_4=1}^{\infty}
		b_{\ell_2;\varepsilon} C_{\ell 0 \ell 0}^{\ell_2 0} b_{\ell_3;\varepsilon}
		C_{\ell 0 \ell 0}^{\ell_3 0} b_{\ell_4;\varepsilon} C_{\ell 0 \ell
			0}^{\ell_4 0} C_{\ell_3 0 \ell_4 0}^{\ell_2 0} \dfrac{1}{\sqrt{(2\ell_2+1)}}%
		\begin{Bmatrix}
		\ell_3 & \ell_4 & \ell_2 \\ 
		\ell & \ell & \ell%
		\end{Bmatrix}
		\\
		& \leq (4\pi)^2 b_{0,\varepsilon} \dfrac{1}{2\ell+1} \sum_{\ell_2,\ell_3,
			\ell_4=1}^{\infty} |b_{\ell_2;\varepsilon} b_{\ell_3;\varepsilon}
		b_{\ell_4;\varepsilon}| \dfrac{\sqrt{2\ell_2+1}}{\sqrt{2\ell+1}}\dfrac{\sqrt{%
				2\ell_3+1}}{\sqrt{2\ell+1}}\dfrac{\sqrt{2\ell_4+1}}{\sqrt{2\ell+1}} 1\times \\
		& \dfrac{1}{\sqrt{(2\ell_2+1)}} \dfrac{1}{\sqrt{2\ell+1}} \min \bigg( \dfrac{%
			1}{\sqrt{2\ell_2+1}}, \dfrac{1}{\sqrt{2\ell_3+1}}, \dfrac{1}{\sqrt{2\ell_4+1}%
		}\bigg) \\
		& \leq (4\pi)^2 b_{0,\varepsilon} \dfrac{1}{(2\ell+1)^3} \sum_{\ell_2,%
			\ell_3, \ell_4=1}^{\infty} |b_{\ell_2;\varepsilon} b_{\ell_3;\varepsilon}
		b_{\ell_4;\varepsilon}| \sqrt{2\ell_3+1} \sqrt{2\ell_4+1}  \times\\
		&  \times\min \bigg( \dfrac{1}{\sqrt{2\ell_2+1}}, \dfrac{1}{\sqrt{2\ell_3+1}}, 
		\dfrac{1}{\sqrt{2\ell_4+1}}\bigg).
		\end{split}
		\end{equation}
		We have already discussed the absolute convergence of the series, which allows us to say that (\ref{third}) is $ O\bigg(\dfrac{1}{\ell^3}\bigg)$
		as $\ell \rightarrow \infty.$ \newline
		\item[-] It remains to study the term $A_{\ell,k}(\ell_1,\ell_2,\ell_3,\ell_4)$;%
		. 
		its absolute value can be bounded by
		\begin{equation}
		\begin{split}
		(4\pi)^2 \sum_{\ell_1=1}^{\infty} |b_{\ell_1;\varepsilon}| &
		\sum_{\ell_2=1}^{\infty} |b_{\ell_2;\varepsilon}| \sum_{\ell_3=1}^{\infty}
		|b_{\ell_3;\varepsilon}| \sum_{\ell_4=1}^{\infty} |b_{\ell_4;\varepsilon}| 
		\dfrac{\sqrt{(2\ell_1+1)(2\ell_2+1)(2\ell_3+1)(2\ell_4+1)}}{(2\ell+1)^2} \times\\
		&\bigg| \sum_{k} C_{\ell_3 0 \ell_4 0}^{k 0}C_{\ell_2 0 \ell_1 0}^{k 0} 
		\begin{Bmatrix}
		\ell & \ell & \ell_1 \\ 
		\ell & \ell & \ell_2 \\ 
		\ell_4 & \ell_3 & k%
		\end{Bmatrix}
		\bigg|.
		\end{split}%
		\end{equation}
		For equation (\ref{20 p.260}) one has that
		\begin{equation}\label{nonsaprei}
		\begin{split}
		\sum_{k} & C_{\ell_3 0 \ell_4 0}^{k 0}C_{\ell_2 0 \ell_1 0}^{k 0} 
		\begin{Bmatrix}
		\ell & \ell & \ell_1 \\ 
		\ell & \ell & \ell_2 \\ 
		\ell_4 & \ell_3 & k%
		\end{Bmatrix}%
		= \dfrac{1}{\sqrt{(2\ell_1+1)(2\ell_2+1)(2\ell_3+1)(2\ell_4+1)}}%
		(-1)^{\ell_1+\ell_2}\times \\
		& \times \sum_{s \sigma} (2s+1) \sqrt{2\ell_1+1}\sqrt{2\ell_3+1} C_{\ell_10 s
			\sigma}^{\ell_4 0} C_{\ell_3 0 s \sigma}^{\ell_2 0} 
		\begin{Bmatrix}
		\ell & \ell & \ell_1 \\ 
		\ell_4 & s & \ell%
		\end{Bmatrix}%
		\begin{Bmatrix}
		\ell & \ell & \ell_3 \\ 
		\ell_2 & s & \ell%
		\end{Bmatrix};
		\end{split}%
		\end{equation}
		the triangular condition implies $\sigma=0,$ so that (\ref{nonsaprei}) is
		\begin{equation}
		\begin{split}
		=&\dfrac{1}{\sqrt{(2\ell_1+1)(2\ell_2+1)(2\ell_3+1)(2\ell_4+1)}} \times \\
		& \times \sum_{s} (2s+1) \sqrt{2\ell_1+1}\sqrt{2\ell_3+1} C_{\ell_10 s
			0}^{\ell_4 0} C_{\ell_3 0 s 0}^{\ell_2 0} 
		\begin{Bmatrix}
		\ell & \ell & \ell_1 \\ 
		\ell_4 & s & \ell%
		\end{Bmatrix}%
		\begin{Bmatrix}
		\ell & \ell & \ell_3 \\ 
		\ell_2 & s & \ell%
		\end{Bmatrix}%
		\end{split}%
		\end{equation}
		and thanks to the fact that

		\begin{equation*}
		\bigg| 
		\begin{Bmatrix}
		\ell & \ell & \ell_1 \\ 
		\ell_4 & s & \ell%
		\end{Bmatrix}%
		\begin{Bmatrix}
		\ell & \ell & \ell_3 \\ 
		\ell_2 & s & \ell%
		\end{Bmatrix}
		\bigg| 
		\leq \dfrac{1}{2\ell+1} \times
		\end{equation*}
		\begin{equation*}
		\times \min \bigg( \dfrac{1}{\sqrt{2\ell_1+1}} ,\dfrac{1}{\sqrt{2s+1}} ,\dfrac{1}{%
			\sqrt{2\ell_4+1}} \bigg) \min \bigg( \dfrac{1}{\sqrt{2\ell_2+1}} ,\dfrac{1}{%
			\sqrt{2s+1}} ,\dfrac{1}{\sqrt{2\ell_3+1}} \bigg)
		\end{equation*}
		((\ref{1 p110}), see \cite{M e Peccati}), and (\ref{star0}), one gets that

		\begin{equation}\label{unaeq}
		\begin{split}
		A_{\ell,k}&(\ell_1,\ell_2,\ell_3,\ell_4) \leq (4\pi)^2 \dfrac{1}{(2\ell+1)^3} \sum_{\ell_1 \ell_2 \ell_3 \ell_4}\!\!\!\!
		|b_{\ell_1;\varepsilon}b_{\ell_2;\varepsilon}b_{\ell_3;\varepsilon}b_{%
			\ell_4;\varepsilon}| \sum_{s} (2s+1) \sqrt{2\ell_1+1} \sqrt{2\ell_3+1} \times\\&
		\times
		C_{\ell_1 0s0}^{\ell_4 0} C_{\ell_3 0s0}^{\ell_2 0}\min \bigg( \dfrac{1}{\sqrt{2\ell_1+1}} ,\dfrac{1}{\sqrt{2s+1}} ,\dfrac{1}{%
			\sqrt{2\ell_4+1}} \bigg) \min \bigg( \dfrac{1}{\sqrt{2\ell_2+1}} ,\dfrac{1}{%
			\sqrt{2s+1}} ,\dfrac{1}{\sqrt{2\ell_3+1}} \bigg)  \\
		\end{split}%
		\end{equation}
		and since
		\begin{equation*}
		\min \bigg( \dfrac{1}{2\ell_1+1} ,\dfrac{1}{\sqrt{2s+1}} ,\dfrac{1}{\sqrt{%
				2\ell_4+1}} \bigg)\leq 1
		\end{equation*}
		and
		\begin{equation*}
		\min \bigg( \dfrac{1}{\sqrt{2\ell_2+1}} ,\dfrac{1}{\sqrt{2s+1}} ,\dfrac{1}{%
			\sqrt{2\ell_3+1}} \bigg)\leq 1,
		\end{equation*}
		then (\ref{unaeq}) is bounded by
		\begin{equation}
		\begin{split}
		\leq (4\pi)^2 \dfrac{1}{(2\ell+1)^3} \sum_{\ell_1 \ell_2 \ell_3 \ell_4}
		|b_{\ell_1;\varepsilon}b_{\ell_2;\varepsilon}b_{\ell_3;\varepsilon}b_{%
			\ell_4;\varepsilon}| \bigg|\sum_{s} (2s+1) \sqrt{2\ell_1+1} \sqrt{2\ell_3+1} C_{\ell_1 0s0}^{\ell_4 0} C_{\ell_3 0s0}^{\ell_2 0} \bigg|.
		\end{split}
		\end{equation}
		In view of the Cauchy-Schwarz inequality, it follows that 
		\begin{equation}\label{nonsaprei2}
		\sum_{s} (2s+1)C_{\ell_1 0s0}^{\ell_4 0} C_{\ell_3 0s0}^{\ell_2 0} \leq %
		\bigg(\sum_{s} (\sqrt{2s+1} C_{\ell_1 0s0}^{\ell_4 0})^2\bigg)^{1/2}\bigg(%
		\sum_{s} (\sqrt{2s+1} C_{\ell_3 0s0}^{\ell_2 0})^2\bigg)^{1/2}
		\end{equation}
		and since
		$\sum_{\ell=|\ell_2-\ell_1|}^{\ell_2+\ell_1} (C_{\ell_10\ell_20}^{\ell 0})^2
		=1,$
		and the permutations properties  \begin{equation*}
		1=\sum_\ell (C_{\ell_10 \ell_2 0}^{\ell 0})^2=\sum_\ell \dfrac{2\ell+1}{%
			2\ell_2+1} (C_{\ell_10 \ell 0}^{\ell_2 0})^2,
		\end{equation*} it entails that 
		\begin{equation*}
		\sum_\ell (2\ell+1)(C_{\ell_10 \ell 0}^{\ell_2 0})^2=2\ell_2+1
		\end{equation*} 
		and then, the right hand side of (\ref{nonsaprei2}) is equal to
		\begin{equation*}
		= \sqrt{2\ell_4+1}\sqrt{2\ell_2+1}.
		\end{equation*}
		Eventually, $A_{\ell,k}(\ell_1,\ell_2,\ell_3,\ell_4)$ is smaller than
		\begin{equation}
		(4\pi)^2 \sum_{\ell_1=1}^{\infty} b_{\ell_1;\varepsilon} C_{\ell 0 \ell
			0}^{\ell_1 0} \sum_{\ell_2=1}^{\infty} b_{\ell_2;\varepsilon} C_{\ell 0 \ell
			0}^{\ell_2 0} \sum_{\ell_3=1}^{\ell} b_{\ell_3;\varepsilon} C_{\ell 0 \ell
			0}^{\ell_3 0} \sum_{\ell_4=1}^{\infty} b_{\ell_4;\varepsilon} C_{\ell 0 \ell
			0}^{\ell_4 0} \sum_{k} C_{\ell_3 0 \ell_4 0}^{k 0}C_{\ell_2 0 \ell_1 0}^{k
			0} 
		\begin{Bmatrix}
		\ell & \ell & \ell_1 \\ 
		\ell & \ell & \ell_2 \\ 
		\ell_4 & \ell_3 & k%
		\end{Bmatrix}%
		\end{equation}
		\begin{equation*}
		\leq (4\pi)^2 \dfrac{1}{(2\ell+1)^3} \sum_{\ell_1 \ell_2 \ell_3 \ell_4}
		|b_{\ell_1;\varepsilon}b_{\ell_2;\varepsilon}b_{\ell_3;\varepsilon}b_{%
			\ell_4;\varepsilon}| \sqrt{2\ell_2+1}\sqrt{2\ell_4+1} \sqrt{2\ell_1+1}\sqrt{%
			2\ell_3+1}
		\end{equation*}
		and since the series are absolutely convergent, the proof is completed.
	\end{itemize}
\end{proof}
\begin{oss}A detailed
	investigation for the asymptotic behavior of Clebsch-Gordan coefficients and bounds like (\ref{nonsaprei3}) and (\ref{nonsaprei})
	can be found also in \cite{M2006} and \cite{M2008}.
\end{oss}

\section{The Clebsch-Gordan coefficients}\label{C-G Appendix}
For completeness, in this Appendix we recall some basic facts and properties about the Clebsch-Gordan coefficients, which we used in our proofs above (we refer to \cite{quantum theory} for further properties and details).\\
The Clebsch-Gordan coefficients are important tools for the evaluation of multiple integrals of spherical harmonics.
For $SO(3)$ they are defined as the set $\{ C_{\ell_1 m_1 \ell_2 m_2}^{\ell_3 m_3} \}$ of the elements of the unitary matrices $C_{\ell_1 \ell_2},$ (Chap 2.4.2 \cite{M e Peccati}); the Clebsch-Gordan coefficients vanish unless the Triangular condition $$|\ell_1-\ell_2|\leq \ell_3 \leq \ell_1+\ell_2,$$
and the equation $$m_1+m_2=m_3$$ are satisfied.\\
The following orthogonal conditions hold (\cite{quantum theory}):

\begin{equation}\label{unitary 1}
\sum_{m_1 m_2} C_{\ell_1 m_1 \ell_2 m_2}^{\ell_3 m_3} C_{\ell_1 m_1 \ell_2 m_2}^{\ell_3' m_3'}=\delta_{\ell_3}^{\ell_3'} \delta_{m_3}^{m_3'},
\end{equation}

\begin{equation}\label{unitary 2}
\sum_{\ell m} C_{\ell_1 m_1 \ell_2 m_2}^{\ell m} C_{\ell_1 m'_1 \ell_2 m'_2}^{\ell m}=\delta_{m_1}^{m'_1} \delta_{m_2}^{m_2'}.
\end{equation}
For $m_1+m_2+m_3=0$, an analytic expression is known:
\begin{equation}
\begin{split}
C_{\ell_1 m_1 \ell_2 m_2}^{\ell_3 -m_3}
&:= (-1)^{\ell_1+m_1} \sqrt{2\ell_3+1} \bigg[ \frac{(\ell_1+\ell_2-\ell_3)!(\ell_1-\ell_2+\ell_3)!(\ell_1-\ell_2+\ell_3)!}{(\ell_1+\ell_2+\ell_3+1)!} \bigg]^{1/2}\\&
\times \bigg[ \frac{(\ell_3+m_3)!(\ell_3-m_3)!}{(\ell_1+m_1)!(\ell_1-m_1)!(\ell_2+m_2)!(\ell_2-m_2)!} \bigg]^{1/2}\\&
\times \sum_{z} \frac{(-1)^z(\ell_2+\ell_3+m_1-z)!(\ell_1-m_1+z)!}{z!(\ell_2+\ell_3-\ell_1-z)!(\ell_3+m_3-z)!(\ell_1-\ell_2-m_3+z)!}
\end{split}
\end{equation}
where the summation runs over all $z$'s such that the factorials are non-negative. When $m_1=m_2=m_3=0,$ this expression becomes simpler

\begin{equation}\label{expression}
C_{\ell_1 0 \ell_2 0 }^{\ell_3 0}=\begin{cases} 0, & \\ \mbox{ }\mbox{ }\mbox{ }\mbox{ }\mbox{ }\mbox{ }\mbox{ }\mbox{ }\mbox{ }\mbox{ }\mbox{ }\mbox{ }\mbox{ }\mbox{ }\mbox{ }\mbox{ }\mbox{ }\mbox{ }\mbox{ }\mbox{ }\mbox{ }\mbox{ }\mbox{ }\mbox{ }\mbox{ }\mbox{ }\mbox{ }\mbox{ }\mbox{ }\mbox{ }\mbox{ }\mbox{ for } \ell_1+\ell_2+\ell_3 \mbox{ odd} \\ \frac{(-1)^{\frac{\ell_1+\ell_2-\ell_3}{2}} [(\ell_1+\ell_2+\ell_3)/2]!}{[(\ell_1+\ell_2-\ell_3)/2]![(\ell_1-\ell_2+\ell_3)/2]![(-\ell_1+\ell_2+\ell_3)/2]!} \big\{ \frac{(\ell_1+\ell_2-\ell_3)!(\ell_1-\ell_2+\ell_3)!(-\ell_1+\ell_2+\ell_3)!}{(\ell_1+\ell_2+\ell_3+1)!}  \big\}^{1/2}, &\\ \mbox{ }\mbox{ }\mbox{ }\mbox{ }\mbox{ }\mbox{ }\mbox{ }\mbox{ }\mbox{ }\mbox{ }\mbox{ }\mbox{ }\mbox{ }\mbox{ }\mbox{ }\mbox{ }\mbox{ }\mbox{ }\mbox{ }\mbox{ }\mbox{ }\mbox{ }\mbox{ }\mbox{ }\mbox{ }\mbox{ }\mbox{ }\mbox{ }\mbox{ }\mbox{ }\mbox{ }\mbox{ for } \ell_1+\ell_2+\ell_3 \mbox{ even} 
\end{cases}
\end{equation}
We recall the following basic property.
\begin{itemize}
	\item Symmetry Properties:

\begin{equation}\label{symmetry p 10}
\begin{split}
C_{\ell_1 m_1 \ell_2 m_2}^{\ell_3 m_3}&=(-1)^{\ell_1+\ell_2-\ell_3} C_{\ell_2 m_2 \ell_1 m_1}^{\ell_3 m_3}=(-1)^{\ell_1-m_1} \sqrt{\dfrac{2\ell_3+1}{2\ell_2+1}} C_{\ell_1 m_1 \ell_3 -m_3}^{\ell_2 -m_2}\\&
=(-1)^{\ell_1-m_1} \sqrt{\dfrac{2\ell_3+1}{2\ell_2+1}} C_{\ell_3 m_3 \ell_1 -m_1}^{\ell_2 m_2}
(-1)^{\ell_2 m_2} \sqrt{\dfrac{2\ell_3+1}{2\ell_1+1}} C_{\ell_3 -m_3 \ell_2 m_2}^{\ell_1 -m_1}\\&
=(-1)^{\ell_2 m_2} \sqrt{\dfrac{2\ell_3+1}{2\ell_1+1}} C_{\ell_2 -m_2 \ell_3 m_3}^{\ell_1 m_1}
\end{split}
\end{equation}

\begin{equation}\label{symmetry p 11}
C_{\ell_1 m_1 \ell_2 m_2}^{\ell_3 m_3}=(-1)^{\ell_1+\ell_2-\ell_3} C_{\ell_1 -m_1 \ell_2 -m_2}^{\ell_3 -m_3}
\end{equation}


\end{itemize}


\subsection{Wigner 3j coefficients}
Wigner 3j coefficients are related to the Clebsch-Gordan coefficients by the identities

\begin{equation}\label{3.67 libro Domenico}
\begin{pmatrix} \ell_1 & \ell_2 & \ell_3 \\ m_1 &
m_2 & m_3\\ 
\end{pmatrix}
= (-1)^{\ell_3+m_3} \dfrac{1}{\sqrt{2\ell_3+1}} C_{\ell_1 -m_1 \ell_2 -m_2}^{\ell_3 m_3}
\end{equation}
\begin{equation}\label{3.68 libro Domenico}
C_{\ell_1 m_1 \ell_2 m_2}^{\ell_3 m_3}=(-1)^{\ell_1-\ell_2+m_3} \sqrt{2\ell_3+1}
\begin{pmatrix} \ell_1 & \ell_2 & \ell_3 \\ m_1 &
m_2 & -m_3\\ 
\end{pmatrix}.
\end{equation}
From \cite{M e Peccati}, we have that for any $\ell_1,\ell_2,\ell_3,$ the following upper bound holds
\begin{equation}
\bigg| \begin{pmatrix} \ell_1 & \ell_2 & \ell_3 \\ m_1 &
m_2 & m_3\\ 
\end{pmatrix}
 \bigg| \leq [\max\{2\ell_1+1,2\ell_2+1,2\ell_3+1\}]^{-1/2},
\end{equation}
then
\begin{equation}\label{mod c}
|C_{\ell_1 m_1 \ell_2 m_2}^{\ell_3 m_3}|\leq \sqrt{2\ell_3+1} [\max\{2\ell_1+1,2\ell_2+1,2\ell_3+1\}]^{-1/2}
\end{equation}
As the Clebsch-Gordan coefficients they satisfy some symmetry properties; such as
\begin{equation}
\begin{pmatrix} \ell_1 & \ell_2 & \ell_3 \\ m_1 &
m_2 & m_3\\ 
\end{pmatrix}
= (-1)^{\ell_1+\ell_2+\ell_3}\begin{pmatrix} \ell_1 & \ell_2 & \ell_3 \\ -m_1 &
-m_2 & -m_3\\ 
\end{pmatrix}.
\end{equation}
For special values of the arguments, namely if $\ell_3=0 \mbox{ or } \ell_2=0$, one has explicit forms of these coefficients:
\begin{equation}\label{1 pag 248}
C_{\ell_1 m_1 \ell_2 m_2}^{ 00} =(-1)^{\ell_1-m_1} \dfrac{\delta_{\ell_1}^{\ell_2} \delta_{m_1}^{-m_2}}{\sqrt{2\ell_1+1}}
\end{equation}
and 
\begin{equation}\label{2 pag 248}
C_{\ell_1 m_1 0 0}^{ \ell_3 m_3} = \delta_{\ell_1}^{\ell_3} \delta_{m_1}^{m_3}  ;
\end{equation}
for details see again \cite{quantum theory}.
Another property, involved in our computations is the sum of the products of four Clebsch-Gordan Coefficients:
\begin{equation}\label{20 p.260}
\sum_{\beta \gamma \epsilon \varphi} C_{b \beta c \gamma }^{a \alpha} C_{e \epsilon f \varphi}^{d \delta} C_{e \epsilon b \beta}^{g \eta} C_{f \varphi c \gamma }^{j \mu} =(-1)^{a-b+c+d+e-f} \sum_{s \sigma} \prod_{s s a g} C_{a \alpha s \sigma}^{j \mu} C_{g \eta s \sigma}^{d \delta} \begin{Bmatrix} b & c & a \\ j &
s & f\\ 
\end{Bmatrix}  \begin{Bmatrix} b & e & g \\ d &
s & f\\ 
\end{Bmatrix}
\end{equation}
$$=\prod_{a d gj} \sum_{k i} C_{g \eta j-\mu}^{k i } C_{d \delta a \alpha }^{k i} \begin{Bmatrix} c & b & a \\ f &
e & d \\ j & g&k \\ 
\end{Bmatrix}$$
(see Section \ref{9j} for the last symbol).
\subsection{Wigner 6j coefficients}
The 6j symbol is invariant under any permutation of its columns or under interchange of the upper and lower arguments in each of any two columns, the formulae we used in the paper are the following:
\begin{equation}\label{2 p.298}
\begin{split}
&\begin{Bmatrix} a & b & c \\ d & e & f \\ \end{Bmatrix}=
\begin{Bmatrix} a & e & f \\ d & b & c \\ \end{Bmatrix} 
=\begin{Bmatrix} e & d & c \\ b & a & f \\ \end{Bmatrix}=
\begin{Bmatrix} d & b & f \\ a & e & c \\ \end{Bmatrix}
\end{split}
\end{equation}
and the following upper bound holds (\cite{M e Peccati}, Section 4.5.4):
\begin{equation}\label{1 p110}
 \bigg| \begin{Bmatrix} a & b & c \\ d &
e & f \\  
\end{Bmatrix}\bigg| \leq \min \bigg( \dfrac{1}{\sqrt{(2c+1)(2f+1)}}, \dfrac{1}{\sqrt{(2a+1)(2d+1)}},\dfrac{1}{\sqrt{(2b+1)(2e+1)}} \bigg)
\end{equation}
When one of the arguments is equal to zero, their expression reduces to
\begin{equation}\label{1 pag.299}
\begin{split}
& \begin{Bmatrix} a & b & c \\ 0 &e & f \\ \end{Bmatrix}= (-1)^{a+b+e} \dfrac{\delta_{b}^{f} \delta_{c}^{e}}{\sqrt{(2b+1)(2c+1)}}, \\&
\begin{Bmatrix} a & 0 & c \\ d &e & f \\ \end{Bmatrix} = (-1)^{a+d+e} \dfrac{\delta_{a}^{c} \delta_{d}^{f}}{\sqrt{(2a+1)(2d+1)}} ,\\&
\begin{Bmatrix} a & b & c \\ d &0 & f \\ \end{Bmatrix}= (-1)^{a+b+d} \dfrac{\delta_{a}^{f} \delta_{c}^{d}}{\sqrt{(2a+1)(2c+1)}}.
\end{split}
\end{equation}

\subsection{Wigner 9j coefficients}\label{9j}
Similarly to the Wigner 6j coefficients, when one of the arguments is equal to zero, they have an easier expression
\begin{equation}\label{1 p.357}
\begin{Bmatrix} a & b & c \\ d &
e & f \\ g & h & 0 \\ 
\end{Bmatrix}=\delta_{c}^{f} \delta_{g}^{h} \dfrac{(-1)^{b+c+d+g}}{[(2c+1)(2g+1)]^{1/2}}  \begin{Bmatrix} a & b & c \\ e &
d & g \\ 
\end{Bmatrix}
\end{equation}
Using symmetry properties, we get
\begin{equation}\label{2 p.357}
\begin{split}
\begin{Bmatrix} 0 & c & c \\ g &
e & b \\ g & d & a \\\end{Bmatrix}& =  \begin{Bmatrix} c& 0 & c \\ d & g & a \\ e & g & b \\  \end{Bmatrix}=  \begin{Bmatrix} g& g & 0 \\ e & d & c \\ b & a & c \\  \end{Bmatrix}= \begin{Bmatrix} g& b & e \\ 0 & c & c \\ g & a & d \\  \end{Bmatrix}= \begin{Bmatrix} a& g & d \\ c & 0 & c \\ b & g & e \\  \end{Bmatrix}= \begin{Bmatrix} b& a & c \\ g & g & 0 \\ e & d & c \\  \end{Bmatrix}= \\&= \begin{Bmatrix} c& e & d \\ c & b & a \\ 0 & g & g \\  \end{Bmatrix}= \begin{Bmatrix} d& c & e \\ a & c & b \\ g & 0 & g \\  \end{Bmatrix}= \begin{Bmatrix} a& b & c \\ d & e & c \\ g & g & 0 \\  \end{Bmatrix} =\dfrac{(-1)^{b+d+c+g}}{[(2c+1)(2g+1)]^{1/2}}  \begin{Bmatrix} a& b & c \\ e & d & g \\   \end{Bmatrix}
\end{split}
\end{equation}

\begin{equation}\label{4 p.358}
\begin{Bmatrix} a & b & c \\ d &
e & f \\ 0 & 0&0 \\ 
\end{Bmatrix}= \dfrac{\delta_{a}^{d} \delta_{b}^{e} \delta_{c}^{f}}{[(2a+1)(2b+1)(2c+1)]^{1/2}};
\end{equation}
to get all the symmetry properties see \cite{quantum theory}; the one used in this paper is
\begin{equation}\label{p.343}
\begin{Bmatrix} a & b & c \\ d &
e & f \\ g& h&j \\ 
\end{Bmatrix}=  \begin{Bmatrix} a & d & g \\ b &
e & h \\ c& f&j \\ 
\end{Bmatrix}.
\end{equation}

\end{document}